\pdfoutput=1
\RequirePackage{ifpdf}
\ifpdf % We are running pdfTeX in pdf mode
\documentclass[pdftex]{sigma}
\else
\documentclass{sigma}
\fi

\usepackage{mathrsfs}

\numberwithin{equation}{section}

\newtheorem{Theorem}{Theorem}[section]
\newtheorem{Corollary}[Theorem]{Corollary}
\newtheorem{Lemma}[Theorem]{Lemma}
\newtheorem{Proposition}[Theorem]{Proposition}
\newtheorem*{claim*}{Claim}
 { \theoremstyle{definition}
\newtheorem{Definition}[Theorem]{Definition}
\newtheorem{Example}[Theorem]{Example}
\newtheorem{Remark}[Theorem]{Remark} }

\newcommand{\op}[1]{\operatorname{#1}}

\newcommand{\acou}[2]{\ensuremath{\left\langle #1 , #2 \right\rangle}}

\newcommand{\scal}[2]{\ensuremath{\left\langle #1 | #2 \right\rangle}}

\newcommand{\Tr}{\ensuremath{\op{Tr}}}

\newcommand{\TR}{\ensuremath{\op{TR}}}

\newcommand{\Res}{\ensuremath{\op{Res}}}

\def\XXint#1#2#3{{\setbox0=\hbox{$#1{#2#3}{\int}$}
\vcenter{\hbox{$#2#3$}}\kern-.5\wd0}}

\newcommand{\C}{\ensuremath{\mathbb{C}}}

\newcommand{\N}{\ensuremath{\mathbb{N}}}

\newcommand{\R}{\ensuremath{\mathbb{R}}}
\newcommand{\bS}{\ensuremath{\mathbb{S}}}
\newcommand{\T}{\ensuremath{\mathbb{T}}}
\newcommand{\Z}{\ensuremath{\mathbb{Z}}}
\newcommand{\CZ}{\ensuremath{\mathbb{C}\setminus\mathbb{Z}}}

\newcommand{\Rn}{\ensuremath{\R^{n}}}

\newcommand{\Ca}[1]{\ensuremath{\mathscr{#1}}}
\newcommand{\cA}{\Ca{A}}

\newcommand{\cH}{\ensuremath{\mathscr{H}}}

\newcommand{\cL}{\ensuremath{\mathscr{L}}}

\newcommand{\cS}{\ensuremath{\mathscr{S}}}

\newcommand{\scL}{\mathscr{L}}

\newcommand{\psido}{$\Psi$DO}
\newcommand{\psidos}{$\Psi$DOs}

\newcommand{\stS}{\mathbb{S}}

\def\dba{{\mathchar'26\mkern-12mu {\rm d}}}
\newcommand{\dbar}{{\, \dba}}

\newcommand{\ord}{{\op{ord}}}

\newcommand{\supp}{\op{supp}}

\newcommand{\bt}{\bullet}

\newcommand{\HH}{\op{HH}}

\newcommand{\subsubset}{\subset\!\subset}

\begin{document}
\allowdisplaybreaks

\newcommand{\arXivNumber}{1911.12164}

\renewcommand{\thefootnote}{}

\renewcommand{\PaperNumber}{061}

\FirstPageHeading

\ShortArticleName{Noncommutative Residue and Canonical Trace on Noncommutative Tori}

\ArticleName{Noncommutative Residue and Canonical Trace\\ on Noncommutative Tori. Uniqueness Results\footnote{This paper is a~contribution to the Special Issue on Noncommutative Manifolds and their Symmetries in honour of Giovanni Landi. The full collection is available at \href{https://www.emis.de/journals/SIGMA/Landi.html}{https://www.emis.de/journals/SIGMA/Landi.html}}}

\Author{Rapha\"el PONGE}

\AuthorNameForHeading{R.~Ponge}

\Address{School of Mathematics, Sichuan University, Chengdu, China}
\Email{\href{mailto:ponge.math@icloud.com}{ponge.math@icloud.com}}
\URLaddress{\url{https://raphaelponge.org}}

\ArticleDates{Received January 09, 2020, in final form June 15, 2020; Published online July 05, 2020}

\Abstract{In this paper we establish uniqueness theorems for the noncommutative residue and the canonical trace on pseudodifferential operators on noncommutative tori of arbitrary dimension. The former is the unique trace up to constant multiple on integer order pseudodifferential operators.The latter is the unique trace up to constant multiple on non-integer order pseudodifferential operators. This improves previous uniqueness results by Fathizadeh--Khalkhali, Fathizadeh--Wong, and L\'evy--Neira--Paycha.}


\Classification{58J42; 58B34; 58J40}

\renewcommand{\thefootnote}{\arabic{footnote}}
\setcounter{footnote}{0}

\section{Introduction}
The noncommutative residue trace of Guillemin~\cite{Gu:AIM85} and Wodzicki~\cite{Wo:HDR, Wo:NCR} is a trace on the algebra of integer order pseudodifferential operators (\psidos) on a closed manifold. Its importance in noncommutative geometry mostly stems from its relationships with the noncommutative integral, the local index formula in noncommutative geometry, and the spectral action principle (see~\cite{Co:NCG, CM:AMS08, CM:GAFA95}). The noncommutative residue appears as the residual functional induced by the analytic extension of the ordinary trace to non-integer \psidos. This analytic extension gives rise to a trace, which is called the canonical trace~\cite{Gu:AIM85, KV:GDEO}. The noncommutative residue and cano\-nical trace provide us with a neat approach to the meromorphic extension of spectral functions attached to elliptic operators such as zeta and eta functions.

A striking feature shared by the noncommutative residue and canonical trace is uniqueness. Namely, by a well known result of Wodzicki~\cite{Wo:HDR} on any connected closed manifold of dimension~$n\geq 2$ the noncommutative residue is the unique trace up to constant multiple on integer order \psidos\ (see also~\cite{Le:AGAG99, LN:JNCG13, LP:PLMS07, Po:JAM10}). This result was extended to Fourier integral operators by Guillemin~\cite{Gu:JFA93} and to \psidos\ with log-polyhomogeneous symbols by Lesch~\cite{Le:AGAG99}. It was also extended to Heisenberg \psidos\ by the author~\cite{Po:JFA07}. The uniqueness of the canonical trace was established by Maniccia--Seiler--Schrohe~\cite{MSS:PAMS08} (see also~\cite{LN:JNCG13, Po:JAM10}).

The noncommutative residue and the canonical trace have been extended to numerous classes of \psidos\ in various geometric settings. This paper focuses on their versions on noncommutative tori. The noncommutative tori are arguably the most well known examples of noncommutative spaces. In particular, 2-dimensional noncommutative tori naturally arise from actions of $\Z$ on the circle $\bS^1$ by irrational rotations. More generally, an $n$-dimensional noncommutative tori $\T^n_\theta$ is generated by unitaries $U_1, \ldots, U_n$ subject to the relations,
\begin{gather*}
 U_lU_j = {\rm e}^{2{\rm i}\pi \theta_{jl}}U_jU_l, \qquad j,l=1,\ldots, n,
\end{gather*}
where $\theta=(\theta_{jl})$ is a given real anti-symmetric matrix. The pseudodifferential calculus on NC tori of Connes~\cite{Ba:CRAS88-I,Ba:CRAS88-II,Co:CRAS80, HLP:Part1, HLP:Part2} has been receiving a great deal of attention recently thanks to its role in the spectral theoretic approach to curvature on noncommutative tori following the seminal work of Connes--Tretkoff~\cite{CT:Baltimore11} and Connes--Moscovici~\cite{CM:JAMS14} (see, e.g.,~\cite{DS:SIGMA15, Fa:JMP15, FK:JNCG12, FK:JNCG13, FK:JNCG15, FGK:JNCG19, IM:JGP18, LM:GAFA16, Liu:JNCG18, Liu:arXiv18b, Liu:arXiv20, Po:JMP20, SZ:arXiv19};
see also~\cite{Co:Survey19, FK:Survey19, LM:Survey19} for recent surveys).

The noncommutative residue for NC tori was defined in dimension $n=2$ and $n=4$ by Fathizadeh--Wong~\cite{FW:JPDOA11} and Fathizadeh--Khalkhali~\cite{FK:JNCG15}, respectively. It was extended to any dimension~$n\geq 2$ by L\'evy--Neira--Paycha~\cite{LNP:TAMS16} who also constructed the canonical trace. We refer to~\cite{Fa:JMP15, Po:JMP20} for the relationship between noncommutative residue and scalar curvature on NC tori.

Uniqueness results for the noncommutative residue and the canonical trace are established under some conditions. It is shown in~\cite{FW:JPDOA11} (resp., \cite{FK:JNCG15}) that the noncommutative residue is the unique \emph{exotic} trace up to constant multiple in dimension~$n\geq 2$ (resp., $n=4$), where by an exotic trace it is meant a trace that vanishes on \psidos\ of sufficiently negative order. This result was subsequently improved in~\cite{LNP:TAMS16}, where it is shown that the noncommutative residue is the unique \emph{singular} trace up to constant multiple in any dimension~$n\geq 2$. In the terminology of~\cite{LNP:TAMS16} a functional on \psidos\ is singular when it vanishes on smoothing operators (and so any exotic trace is singular).

As for the canonical trace, in~\cite{LNP:TAMS16} it is also shown to be the unique \emph{$\scL^1$-continuous} trace up to constant multiple on non-integer order \psidos.
A functional on \psidos\ is $\scL^1$-continuous when its restriction to \psidos\ of order~$<-n$ is continuous with respect to the topology induced by the trace-class norm.

The aim is of this paper is to remove all the above restrictions on the trace classes in the aforementioned uniqueness theorems. Our first uniqueness result states that the noncommutative residue is the unique trace on integer order \psidos\ on NC tori of any dimension~$n\geq 2$ (Theorem~\ref{thm:Uniqueness.NCR}). This implies that the first Hochschild cohomology group has dimension~1 (Corollary~\ref{cor:Uniqueness.HH0}). Our second uniqueness result asserts that
 the canonical trace is the unique trace on non-integer order \psidos\ on NC tori of any dimension~$n\geq 2$ (Theorem~\ref{thm:Uniqueness.TR}).
As a by-product the canonical trace is the unique trace on non-integer order \psidos\ that agrees with the ordinary trace on smoothing operators (Corollary~\ref{cor:Uniqueness.TR-Tr}).

The approach of this paper is similar to the approaches used in~\cite{Le:AGAG99, Po:JFA07, Po:JAM10}. In particular, we show that the noncommutative residue (resp., the canonical trace) is the unique obstruction to being a sum of integer order (resp., non-integer order) \psido\ commutators (see Propositions~\ref{prop:Uniqueness.NCR-Com} and~\ref{prop:Uniqueness.TR-Com} for the precise statements). These results are obtained in two main steps. In the first step we show that a \psido\ is a sum of commutators modulo smoothing operators if and only if either it has non-integer order, or it has integer order and zero noncommutative residue (Proposition~\ref{Prop3}). This step is enough to recover the uniqueness results of~\cite{FK:JNCG15, FW:JPDOA11, LNP:TAMS16} for the noncommutative residue (see Corollary~\ref{cor:Comm-mod-smooth.unique-singular-trace}).

The second step is twofold. First, we apply a result of Guillemin~\cite{Gu:JFA93} to show that any smoothing operator with zero trace is the sum of two commutators of smoothing operators (Proposition~\ref{lemF}). Combining this with the first step gives the uniqueness of the canonical trace. Second, we show that any smoothing operator is a sum of commutators of integer order \psidos\ (Proposition~\ref{Prop4}). Combining this with the first step establishes the uniqueness of the noncommutative residue.

The noncommutative residue of~\cite{FK:JNCG15, FW:JPDOA11} is defined in dimension $n=2$ and $n=4$ only. In~\cite{LNP:TAMS16} the noncommutative residue and the canonical trace are defined in any dimension. However, the authors actually consider toroidal \psidos.
We obtain traces on our class of \psidos\ by using the equivalence between these classes established in~\cite{HLP:Part1}.
Furthermore, the arguments of~\cite{LNP:TAMS16} involves considerations on extension maps that are not needed for our purpose.
For sake of simplicity we present a direct construction of the noncommutative residue and canonical trace for our class of \psidos\ in Section~\ref{sec:NCR-TR}.

In addition, the proof of the trace property of the canonical trace in~\cite{LNP:TAMS16} is not complete (see Remark~\ref{rmk:Uniqueness.TR-proof} on this point). For sake of completeness we provide a proof of this result by another approach at the end of Section~\ref{sec:NCR-TR}. The approach is similar to the classical approach of~\cite{Gu:AIM85, KV:GDEO}.

The paper is organized as follows. In Section~\ref{sec:NCtori}, we review the main definitions and properties of noncommutative tori.
In Section~\ref{sec:PsiDOs}, we recall the main definitions and properties of pseudo\-differential operators on noncommutative tori.
In Section~\ref{sec:NCR-TR}, we review the constructions of the noncommutative residue on integer-order \psidos\ and of the canonical trace on non-integer-order \psidos\ on NC tori.
In Section~\ref{sec:commutators-mod-smoothing}, we derive the decomposition of \psidos\ as sums of \psido\ commutators modulo smoothing operators.
In Section~\ref{sec:commutator-smoothing}, we study the commutator structure of smoothing operators.
In Section~\ref{sec:Uniqueness}, we prove our main uniqueness results for the noncommutative residue and the canonical trace.
The last two sections, Section~\ref{sec:lemD} and Section~\ref{sec:lemN5}, are devoted to the proofs of Lemmas~\ref{lemD} and~\ref{lemN5}, respectively.

It would be interesting to try to extend the results of this paper in two directions. The first direction concerns traces on log-polyhomogeenous symbols on NC tori. We should expect to have analogues for NC tori of the results of Lesch~\cite{Le:AGAG99}. The other direction concerns pre-traces and hyper-traces on \psidos\ of a fixed given order (see~\cite{LN:JNCG13} for the corresponding results on ordinary manifolds). There are results in this direction in~\cite{LNP:TAMS16}, but they are established under some continuity conditions that it would be desirable to remove.

\section{Noncommutative tori} \label{sec:NCtori}
In this section, we review the main definitions and properties of noncommutative $n$-tori, $n\geq 2$. We refer to~\cite{Co:NCG, HLP:Part1,Ri:PJM81, Ri:CM90}, and the references therein, for a more comprehensive account.

Throughout this paper, we let $\theta =(\theta_{jk})$ be a real anti-symmetric $(n\times n)$-matrix. Denote by $\theta_1, \ldots, \theta_n$ its column vectors. We also let $L^2(\T^n)$ be the Hilbert space of $L^2$-functions on the ordinary torus $\T^n=\R^n\slash 2\pi \Z^n$ equipped with the inner product,
\begin{gather} \label{eq:NCtori.innerproduct-L2}
 \scal{\xi}{\eta}= (2\pi)^{-n} \int_{\T^n} \xi(x)\overline{\eta(x)}\dbar x, \qquad \xi, \eta \in L^2(\T^n).
\end{gather}
 For $j=1,\ldots, n$, let $U_j\colon L^2(\T^n)\rightarrow L^2(\T^n)$ be the unitary operator defined by
 \begin{gather*}
( U_j\xi)(x)= {\rm e}^{{\rm i}x_j} \xi( x+\pi \theta_j), \qquad \xi \in L^2(\T^n).
\end{gather*}
 We then have the relations,
 \begin{gather} \label{eq:NCtori.unitaries-relations}
 U_kU_j = {\rm e}^{2{\rm i}\pi \theta_{jk}} U_jU_k, \qquad j,k=1, \ldots, n.
\end{gather}

The \emph{noncommutative torus} $A_\theta=C(\T^n_\theta)$ is the $C^*$-algebra generated by the unitary operators $U_1, \ldots, U_n$. For $\theta=0$ we obtain the $C^*$-algebra $C(\T^n)$ of continuous functions on the ordinary $n$-torus $\T^n$. Note that~(\ref{eq:NCtori.unitaries-relations}) implies that $A_\theta$ is the closure in $\cL\big(L^2(\T^n)\big)$ of the linear span of the unitary operators,
 \begin{gather*}
 U^k:=U_1^{k_1} \cdots U_n^{k_n}, \qquad k=(k_1,\ldots, k_n)\in \Z^n.
\end{gather*}

 Let $\tau\colon \cL(L^2(\T^n))\rightarrow \C$ be the state defined by the constant function~$1$, i.e.,
 \begin{gather*}
 \tau (T)= \scal{T1}{1}=\int_{\T^n} (T1)(x) \dbar x, \qquad T\in \cL\big(L^2(\T^n)\big).
\end{gather*}
This induces a continuous tracial state on the $C^*$-algebra $A_\theta$ such that $\tau(1)=1$ and $\tau\big(U^k\big)=0$ for $k\neq 0$.
 The GNS construction (see, e.g., \cite{Ar:Springer81}) then allows us to associate with $\tau$ a~$*$-re\-presentation of $A_\theta$ as follows.

 Let $\scal{\cdot}{\cdot}$ be the sesquilinear form on $A_\theta$ defined by
\begin{gather}
 \scal{u}{v} = \tau\left( uv^* \right), \qquad u,v\in A_\theta.
 \label{eq:NCtori.cAtheta-innerproduct}
\end{gather}
Note that the family $\big\{ U^k; \, k \in \Z^n\big\}$ is orthonormal with respect to this sesquilinear form. We let $\cH_\theta=L^2(\T^n_\theta)$ the Hilbert space arising from the completion of $A_\theta$ with respect to the pre-inner product~(\ref{eq:NCtori.cAtheta-innerproduct}). The action of $A_\theta$ on itself by left-multiplication uniquely extends to a~$*$-representation of~$A_\theta$ in $\cH_\theta$. When $\theta=0$ we recover the Hilbert space $L^2(\T^n)$ with the inner product~(\ref{eq:NCtori.innerproduct-L2}) and the representation of $C(\T^n)$ by bounded multipliers. In addition, as $(U^k)_{k \in \Z^n}$ is an orthonormal basis of $\cH_\theta$, every $u\in \cH_\theta$ can be uniquely written as
\begin{gather} \label{eq:NCtori.Fourier-series-u}
 u =\sum_{k \in \Z^n} u_k U^k, \qquad u_k:=\big\langle u\,|\,U^k\big\rangle,
\end{gather}
where the series converges in $\cH_\theta$. When $\theta =0$ we recover the Fourier series decomposition in $L^2(\T^n)$.

The natural action of $\R^n$ on $\T^n$ by translation gives rise to an action on $\cL\big(L^2(\T^n)\big)$. This induces a $*$-action $(s,u)\rightarrow \alpha_s(u)$ on $A_\theta$ given by
\begin{gather*}
\alpha_s(U^k)= {\rm e}^{{\rm i}s\cdot k} U^k, \qquad \text{for all $k\in \Z^n$ and $s\in \R^n$}.
\end{gather*}
This action is strongly continuous, and so we obtain a $C^*$-dynamical system $\big(A_\theta, \R^n, \alpha\big)$. We are especially interested in the subalgebra $\cA_\theta=C^\infty(\T^n_\theta)$ of smooth elements of this $C^*$-dynamical system (a.k.a.~\emph{smooth noncommutative torus}). Namely,
\begin{gather*}
 \cA_\theta:=\big\{ u \in A_\theta; \, \alpha_s(u) \in C^\infty\big(\R^n; A_\theta\big)\big\}.
\end{gather*}
The unitaries $U^k$, $k\in \Z^n$, are contained in $\cA_\theta$, and so $\cA_\theta$ is a dense subalgebra of $A_\theta$. Denote by~$\cS\big(\Z^n\big)$ the space of rapid-decay sequences with complex entries. In terms of the Fourier series decomposition~(\ref{eq:NCtori.Fourier-series-u}) we have
\begin{gather*}
 \cA_\theta=\bigg\{ u=\sum_{k\in \Z^n} u_k U^k;\, (u_k)_{k\in \Z^n}\in \cS\big(\Z^n\big)\bigg\}.
\end{gather*}
When $\theta=0$ we recover the algebra $C^\infty(\T^n)$ of smooth functions on the ordinary torus $\T^n$ and the Fourier-series description of this algebra.

For $j=1,\ldots, n$, let $\delta_j\colon \cA_\theta\rightarrow \cA_\theta$ be the derivation defined by
\begin{gather*}
 \delta_j(u) = D_{s_j} \alpha_s(u)|_{s=0}, \qquad u\in \cA_\theta,
\end{gather*}
where we have set $D_{s_j}=\frac{1}{i}\partial_{s_j}$. When $\theta=0$ the derivation $\delta_j$ is just the derivation $D_{x_j}=\frac{1}{i}\frac{\partial}{\partial x_j}$ on $C^\infty(\T^n)$. In general, for $j,l=1,\ldots, n$, we have
\begin{gather*}
 \delta_j(U_l) =
 \begin{cases}
 U_j & \text{if $l=j$},\\
 0 & \text{if $l\neq j$}.
\end{cases}
\end{gather*}
More generally, given any multi-order $\beta \in \N_0^n$, define
\begin{gather*}
 \delta^\beta(u) = D_s^\beta \alpha_s(u)|_{s=0} = \delta_1^{\beta_1} \cdots \delta_n^{\beta_n}(u), \qquad u\in \cA_\theta.
\end{gather*}
We endow $\cA_\theta$ with the locally convex topology defined by the semi-norms,
\begin{gather*}
 \cA_\theta \ni u \longrightarrow \big\|\delta^\beta (u)\big\| , \qquad \beta\in \N_0^n.
%\label{eq:NCtori.cAtheta-semi-norms}
\end{gather*}
With the involution inherited from $A_\theta$ this turns $\cA_\theta$ into a (unital) Fr\'echet $*$-algebra. The Fourier series~(\ref{eq:NCtori.Fourier-series-u}) of every $u\in \cA_\theta$ converges in $\cA_\theta$ with respect to this topology. In addition, it can be shown that $\cA_\theta$ is closed under holomorphic functional calculus (see, e.g., \cite{Co:AIM81, HLP:Part1}).

\section{Pseudodifferential operators on noncommutative tori} \label{sec:PsiDOs}
In this section, we recall the main definitions and properties of pseudodifferential operators on noncommutative tori~\cite{Ba:CRAS88-I,Ba:CRAS88-II, Co:CRAS80, HLP:Part1, HLP:Part2}. The exposition closely follows~\cite{HLP:Part1, HLP:Part2} (see also~\cite{Ta:JPCS18}).

\subsection{Symbols on noncommutative tori}
There are various classes of symbols on noncommutative tori.

\begin{Definition}[standard symbols; see~\cite{Ba:CRAS88-I,Ba:CRAS88-II, Co:CRAS80}]
$\stS^m \big(\Rn ; \cA_\theta\big)$, $m\in\R$, consists of maps $\rho(\xi)\in C^\infty \big(\Rn ; \cA_\theta\big)$ such that, for all multi-orders $\alpha$ and $\beta$, there exists $C_{\alpha \beta} > 0$ such that
\begin{gather*}
\big\|\delta^\alpha \partial_\xi^\beta \rho(\xi)\big\| \leq C_{\alpha \beta} ( 1 + | \xi | )^{m - | \beta |} \qquad \forall\, \xi \in \R^n .
\end{gather*}
\end{Definition}

\begin{Definition}
 $\cS\big(\R^n; \cA_\theta\big)$ consists of maps $\rho(\xi)\in C^\infty \big(\Rn ; \cA_\theta\big)$ such that, for all $N\geq 0$ and multi-orders $\alpha$, $\beta$, there exists $C_{N\alpha \beta} > 0$ such that
\begin{gather*}
\big\|\delta^\alpha \partial_\xi^\beta \rho(\xi)\big\| \leq C_{N\alpha \beta} ( 1 + | \xi |)^{-N} \qquad \forall\, \xi \in \R^n .
\end{gather*}
\end{Definition}

\begin{Remark}\label{rmk:Symbols.symbols-intersection}
$\cS\big(\R^n; \cA_\theta\big)=\bigcap\limits_{m\in\R}\stS^m\big( \R^n;\cA_\theta\big)$.
\end{Remark}

\begin{Definition}[homogeneous symbols]
$S_q \big(\R^n; \cA_\theta \big)$, $q \in \C$, consists of maps \[ \rho(\xi) \in C^\infty\big(\R^n\backslash 0;\cA_\theta\big)\] that are homogeneous of degree $q$, i.e.,
$\rho( t \xi ) =t^q \rho(\xi)$ for all $\xi \in \R^n \backslash 0$ and $t > 0$.
\end{Definition}

\begin{Remark}
 If $\rho(\xi)\in S_q\big(\R^n;\cA_\theta\big)$ and $\chi(\xi)\in C^\infty_c\big(\R^n\big)$ is such that $\chi(\xi)=1$ near $\xi=0$, then $(1-\chi(\xi))\rho(\xi)\in \stS^{\Re q}\big( \R^n;\cA_\theta\big)$.
\end{Remark}

\begin{Definition}[classical symbols; see \cite{Ba:CRAS88-I,Ba:CRAS88-II}]\label{def:Symbols.classicalsymbols}
$S^q \big(\R^n; \cA_\theta \big)$, $q \in \C$, consists of maps $\rho(\xi)\in C^\infty\big(\R^n;\cA_\theta\big)$ that admit an asymptotic expansion,
\begin{gather*}
\rho(\xi) \sim \sum_{j \geq 0} \rho_{q-j} (\xi), \qquad \rho_{q-j} \in S_{q-j} \big(\R^n; \cA_\theta \big),
\end{gather*}
where $\sim$ means that, for all $N\geq 0$ and multi-orders $\alpha$, $\beta$, there exists $C_{N\alpha\beta} >0$ such that, for all $\xi \in \R^n$ with $| \xi | \geq 1$, we have
\begin{gather*} %\label{eq:Symbols.classical-estimates}
\bigg\| \delta^\alpha \partial_\xi^\beta \bigg( \rho - \sum_{j<N} \rho_{q-j} \bigg)(\xi)\bigg\| \leq C_{N\alpha\beta} | \xi |^{\Re{q}-N-| \beta |} .
\end{gather*}
\end{Definition}

\begin{Remark} \label{rmk:Symbols.classical-inclusion}
$S^q\big(\R^n;\cA_\theta\big)\subset \stS^{\Re{q}}\big(\R^n;\cA_\theta\big)$.
\end{Remark}

\begin{Example} Any polynomial map $\rho(\xi)=\sum\limits_{|\alpha|\leq m} a_\alpha \xi^\alpha$, $a_\alpha\in \cA_\theta$, $m\in \N_0$, is in $S^m\big(\R^n;\cA_\theta\big)$.
\end{Example}

\begin{Remark} It is also convenient to consider scalar-valued symbols. We shall denote by $\stS^m\big(\R^n\big)$, $\cS\big(\R^n\big)$, $S_q\big(\R^n\big)$ and $S^q\big(\R^n\big)$ the corresponding classes of scalar-valued symbols. We will consider them as sub-classes of the $\cA_\theta$-valued symbol classes via the natural embedding of~$\C$ into~$\cA_\theta$.
\end{Remark}

\subsection[$\Psi$DOs on noncommutative tori]{$\boldsymbol{\Psi}$DOs on noncommutative tori}

Given $\rho(\xi)\in \stS^m\big(\R^n;\cA_\theta\big)$, $m\in \R$, we let $P_\rho\colon \cA_\theta \rightarrow \cA_\theta$ be the operator defined by
\begin{gather*} %\label{eq:PsiDOs.PsiDO-definition}
P_\rho u = \iint {\rm e}^{{\rm i}s\cdot\xi}\rho(\xi)\alpha_{-s}(u){\rm d}s\dbar\xi, \qquad u \in \cA_\theta.
\end{gather*}
The above integral is meant as an oscillating integral (see~\cite{HLP:Part1}). Equivalently, for all $u=\sum\limits_{k\in \Z^n} u_k U^k$ in $\cA_\theta$, we have
\begin{gather} \label{eq:toroidal.Prhou-equation}
 P_{\rho}u = \sum_{k\in \Z^n} u_k \rho(k)U^k.
\end{gather}
In any case, this defines a continuous linear operator $P_\rho\colon \cA_\theta \rightarrow \cA_\theta$.

\begin{Definition}
$\Psi^q(\T^n_\theta)$, $q\in \C$, consists of all linear operators $P_\rho\colon \cA_\theta\rightarrow \cA_\theta$ with $\rho(\xi)$ in $S^q\big(\R^n; \cA_\theta\big)$.
\end{Definition}

\begin{Remark} \label{rem:PsiDOs.symbol-uniqueness}
If $P=P_\rho$ with $\rho(\xi)$ in $S^q\big(\R^n; \cA_\theta\big)$, $\rho(\xi)\sim \sum \rho_{q-j}(\xi)$, then $\rho(\xi)$ is called a \emph{symbol} for $P$. This symbol is not unique, but its restriction to $\Z^n$ and its class modulo $\cS\big(\R^n;\cA_\theta\big)$ are unique (see Remark~\ref{rmk:PsiDOs.rho(k)-P}). As a result, the homogeneous symbols $\rho_{q-j}(\xi)$ are uniquely determined by~$P$. The leading symbol $\rho_q(\xi)$ is called the \emph{principal symbol} of~$P$.
\end{Remark}

\begin{Example}A differential operator on $\cA_\theta$ is of the form $P=\sum\limits_{|\alpha|\leq m}a_\alpha\delta^\alpha$, $a_\alpha\in\cA_\theta$ (see~\cite{Co:CRAS80, Co:NCG}). This is a \psido\ of order~$m$ with symbol $\rho(\xi)= \sum a_\alpha \xi^\alpha$ (see~\cite{HLP:Part1}).
\end{Example}

\subsection[Composition of $\Psi$DOs]{Composition of $\boldsymbol{\Psi}$DOs}

Suppose we are given symbols $\rho_1(\xi)\in\stS^{m_1}\big(\Rn;\cA_\theta\big)$, $m_1\in\R$, and $\rho_2(\xi)\in\stS^{m_2}\big(\Rn;\cA_\theta\big)$, $m_2\in\R$. As $P_{\rho_1}$ and $P_{\rho_2}$ are linear operators on $\cA_\theta$, the composition $P_{\rho_1}P_{\rho_2}$ makes sense as such an operator.
In addition, we define the map $\rho_1\sharp\rho_2\colon \Rn\rightarrow \cA_\theta$ by
\begin{gather*} %\label{eq:Composition.symbol-sharp}
\rho_1\sharp\rho_2(\xi) = \iint {\rm e}^{{\rm i}t\cdot\eta}\rho_1(\xi+\eta)\alpha_{-t}[\rho_2(\xi)]{\rm d}t\dbar\eta , \qquad \xi\in\Rn ,
\end{gather*}
where the above integral is meant as an oscillating integral (see~\cite{Ba:CRAS88-I,Ba:CRAS88-II, HLP:Part2}).

\begin{Proposition}[see \cite{Ba:CRAS88-I,Ba:CRAS88-II, Co:CRAS80, HLP:Part2}] \label{prop:Composition.sharp-continuity-standard-symbol}
Let $\rho_1(\xi)\in \stS^{m_1}\big(\R^n; \cA_\theta\big)$ and $\rho_2(\xi)\in \stS^{m_2}\big(\R^n; \cA_\theta\big)$, $m_1,m_2\in \R$.
\begin{enumerate}\itemsep=0pt
 \item[$1.$] $\rho_1\sharp\rho_2(\xi)\in \stS^{m_1+m_2}\big(\Rn;\cA_\theta\big)$, and we have $ \rho_1\sharp\rho_2(\xi) \sim \sum\frac{1}{\alpha !}\partial_\xi^\alpha\rho_1(\xi)\delta^\alpha\rho_2(\xi)$.

 \item[$2.$] The operators $P_{\rho_1}P_{\rho_2}$ and $P_{\rho_1\sharp \rho_2}$ agree.
\end{enumerate}
\end{Proposition}

In the case of classical \psidos\ we further have the following result.

\begin{Proposition}[see \cite{Ba:CRAS88-I,Ba:CRAS88-II, Co:CRAS80, HLP:Part2}] \label{prop:Composition.composition-PsiDOs}
Let $P_1\!\in \!\Psi^{q_1}(\cA_\theta)$, $q_1\in \C$, have symbol $\rho_1(\xi)\!\sim\!\sum\limits_{j\geq 0}\rho_{1,q_1-j}(\xi)$, and let $P_2\in \Psi^{q_2}(\cA_\theta)$, $q_2\in \C$, have symbol $\rho_2(\xi)\sim\sum\limits_{j\geq 0}\rho_{2,q_2-j}(\xi)$.
\begin{enumerate}\itemsep=0pt
 \item[$1.$] $\rho_1 \sharp \rho_2(\xi)\in S^{q_1+q_2}\big(\R^n; \cA_\theta\big)$ with $\rho_1\sharp\rho_2(\xi) \sim\sum (\rho_1\sharp\rho_2)_{q_1+q_2-j}(\xi)$, where
 \begin{gather*} %\label{eq:Composition-Sym.sharp-homo-asymptotics}
(\rho_1\sharp\rho_2)_{q_1+q_2-j}(\xi)=\sum_{k+l+|\alpha|=j}\frac{1}{\alpha !}\partial_\xi^\alpha \rho_{1,q_1-k}(\xi)\delta^\alpha \rho_{2,q_2-l}(\xi), \qquad j\geq 0.
\end{gather*}

\item[$2.$] The composition $P_1P_2= P_{\rho_1 \sharp \rho_2}$ is contained in $\Psi^{q_1+q_2}(\cA_\theta)$.
\end{enumerate}
\end{Proposition}

\subsection[Toroidal symbols and $\Psi$DOs]{Toroidal symbols and $\boldsymbol{\Psi}$DOs}\label{subsec:toroidal}
In this section, we recall the relationship between standard \psidos\ as defined above and the toroidal \psidos\ considered in~\cite{GJP:MAMS17, LNP:TAMS16, RT:Birkhauser10}.

Let $\cA_\theta^{\Z^n}$ denote the space of sequences with values in $\cA_\theta$ that are indexed by $\Z^n$. In addition, let $(e_1, \ldots, e_n)$ be the canonical basis of $\R^n$. For $i=1,\ldots, n$, the (forward) difference operator $\Delta_i\colon \cA_\theta^{\Z^n} \rightarrow \cA_\theta^{\Z^n}$ is defined by
\begin{gather*}
 \Delta_i u_k= u_{k+e_i}-u_k, \qquad (u_k)_{k\in \Z^n} \in \cA_\theta^{\Z^n}.
\end{gather*}
The operators $\Delta_1, \ldots, \Delta_n$ pairwise commute. For $\alpha =(\alpha_1, \ldots, \alpha_n)\in \N_0^n$, we set $\Delta^\alpha =\Delta_1^{\alpha_1} \cdots \Delta_n^{\alpha_n}$. We similarly define backward difference operators $\overline{\Delta}=\overline{\Delta}_1^{\alpha_1} \cdots \overline{\Delta}_n^{\alpha_n}$, where $\overline{\Delta}_i\colon \cA_\theta^{\Z^n}\! \allowbreak \rightarrow \cA_\theta^{\Z^n}$ is defined by
\begin{gather*}
 \overline{\Delta}_i u_k= u_{k}-u_{k-e_i}, \qquad (u_k)_{k\in \Z^n} \in \cA_\theta^{\Z^n}.
 %\label{eq:PsiDOs.backward-difference}
\end{gather*}

We will also denote by $\Delta^\alpha$ and $\overline{\Delta}^\alpha$ the extensions of these operators to operators on $\cA_\theta^{\R^n}$, where $\cA_\theta^{\R^n}$ is the space of maps from $\R^n$ to $\cA_\theta$. Namely, we set
$\Delta^\alpha =\Delta_1^{\alpha_1} \cdots \Delta_n^{\alpha_n}$ and $\overline{\Delta}=\overline{\Delta}_1^{\alpha_1} \cdots \overline{\Delta}_n^{\alpha_n}$, where $\Delta_j\colon \cA_\theta^{\R^n}\rightarrow \cA_\theta^{\R^n}$ and $\overline{\Delta}_j\colon \cA_\theta^{\R^n}\rightarrow \cA_\theta^{\R^n}$ are given by
\begin{gather}
 \Delta_j\rho(x)= \rho(\xi+e_j)-\rho(\xi), \qquad \overline{\Delta}_j\rho(x)= \rho(\xi)-\rho(\xi-e_j), \qquad \rho(\xi)\in \cA_\theta^{\R^n}.
 \label{eq:PsiDOs.finite-differences-functions}
\end{gather}

\begin{Definition}[\cite{LNP:TAMS16, RT:Birkhauser10}] $\stS^m\big(\Z^n; \cA_\theta\big)$, $m\in \R$, consists of sequences $(\rho_k)_{k\in \Z^n} \subset\cA_\theta$ such that, for all multi-orders $\alpha$ and $\beta$, there is $C_{\alpha\beta}>0$ such that
\begin{gather*}
 \big\| \delta^\alpha \Delta^\beta \rho_k \big\| \leq C_{\alpha\beta} (1+|k|)^{m-|\beta|} \qquad \forall\, k \in \Z^n.
\end{gather*}
\end{Definition}

\begin{Definition}
 $\cS\big(\Z^n; \cA_\theta\big)$ consists of sequences $(\rho_k)_{k\in \Z^n} \subset \cA_\theta$ such that, for all $N\in \N_0$ and $\alpha\in \N_0^n$, there is $C_{N\alpha}>0$ such that
 \begin{gather*}
 \|\delta^\alpha \rho_k\|\leq C_{N\alpha}(1+|k|)^{-N}\qquad \forall\, k \in \Z^n.
\end{gather*}
\end{Definition}

\begin{Remark} $\cS(\Z^n; \cA_\theta) = \bigcap\limits_{m\in \R} \stS^m\big(\Z^n; \cA_\theta\big)$.
\end{Remark}

 Following~\cite{LNP:TAMS16} (see also~\cite{GJP:MAMS17}) the \psido\ associated with a toroidal symbol $(\rho_k)_{k{\in}\Z^n}\!\!\in\! \stS^m\big(\Z^n;\! \cA_\theta\big)\!,$ $m\in \R$, is defined as the linear operator $P\colon \cA_\theta \rightarrow \cA_\theta$ given by
\begin{gather*}
 P u = \sum_{k\in \Z^n} u_k\rho_kU^k, \qquad u=\sum u_kU^k\in \cA_\theta.
\end{gather*}
Any standard \psido\ is a toroidal \psido. This is seen by using~(\ref{eq:toroidal.Prhou-equation}) and the following result.

\begin{Proposition}[see \cite{HLP:Part1, RT:Birkhauser10}] \label{prop:toroidal.restriction-of-standard-symbol-is-toroidal-symbol}
 Let $\rho(\xi)\in \stS^m\big(\R^n;\cA_\theta\big)$, $m\in \R$.
\begin{enumerate}\itemsep=0pt
 \item[$1.$] The restriction of $\rho(\xi)$ to $\Z^n$ is a toroidal symbol $(\rho(k))_{k\in \Z^n}$ in $\stS^m\big(\Z^n; \cA_\theta\big)$.

 \item[$2.$] If $(\rho(k))_{k\in \Z^n}\in \cS\big(\Z^n; \cA_\theta\big)$, then $\rho(\xi)\in \cS\big(\R^n;\cA_\theta\big)$.
\end{enumerate}
\end{Proposition}

\begin{Remark}\label{rmk:PsiDOs.rho(k)-P}
 If $\rho(\xi) \in \stS^m(\cA_\theta)$, then~(\ref{eq:toroidal.Prhou-equation}) implies that $P_\rho\big(U^k\big)=\rho(k)U^k$, i.e., we have
\begin{gather*}
 \rho(k)=P_\rho\big(U^k\big)\big(U^{k}\big)^{-1}\qquad \text{for all $k\in \Z^n$}.
\end{gather*}
Thus, the restriction to $\Z^n$ of $\rho(\xi)$ is uniquely determined by $P_\rho$. Combining this with the 2nd part of Proposition~\ref{prop:toroidal.restriction-of-standard-symbol-is-toroidal-symbol} shows that the class of $\rho(\xi)$ modulo $\cS\big(\R^n;\cA_\theta\big)$ is uniquely determined by $P_\rho$.
\end{Remark}

Conversely, toroidal symbols can be extended to standard symbols as follows.

\begin{Lemma}[{\cite[Lemma~4.5.1]{RT:Birkhauser10}}] \label{lem:toroidal.phi} There exists a function $\phi(\xi)\in \cS\big(\Rn\big)$ such that
\begin{enumerate}\itemsep=0pt
\item[$(i)$] $\phi(0)=1$ and $\phi(k)=0$ for all $k\in \Z^n\setminus 0$.

\item[$(ii)$] For every multi-order $\alpha$, there is $\phi_\alpha(\xi)\in \cS\big(\Rn\big)$ such that $\partial_\xi^\alpha \phi(\xi) =\overline{\Delta}^\alpha \phi_\alpha(\xi)$.

\item[$(iii)$] $\int \phi(\xi)=1$.
\end{enumerate}
\end{Lemma}

Let $\phi(\xi)\in \cS\big(\Rn\big)$ be a function satisfying the properties (i)--(iii) of Lemma~\ref{lem:toroidal.phi}. Given any toroidal symbol $(\rho_k)_{k\in \Z^n}$ in $\stS^m\big(\Z^n;\cA_\theta\big)$, $m\in \R$, we set
\begin{gather}
\tilde{\rho}(\xi) = \sum_{k\in \Z^n} \phi(\xi-k) \rho_k, \qquad \xi \in \R^n.
\label{eq:toroidal.trho}
\end{gather}
Given any $k\in \Z^n$, we have
 $ \tilde{\rho}(k)= \sum\limits_{\ell \in \Z^n} \phi(k-\ell) \rho_{\ell}= \sum\limits_{\ell \in \Z^n} \delta_{\ell,k} \rho_{\ell}=\rho_k$. Therefore, this provides us with an extension map from $\stS^m\big(\Z^n;\cA_\theta\big)$ to $\stS^m\big(\R^n;\cA_\theta\big)$.

\begin{Proposition}[see~\cite{HLP:Part1,LNP:TAMS16, RT:Birkhauser10}] \label{prop:toroidal.extension-of-toroidal-symbol-is-standard-symbol}
 Let $(\rho_k)_{k\in \Z^n}\in \stS^m\big(\Z^n;\cA_\theta\big)$, $m\in \R$, and denote by $P$ the corresponding toroidal \psido.
\begin{enumerate}\itemsep=0pt
\item[$1.$] $\tilde{\rho}(\xi)$ is a standard symbol in $\stS^m\big(\R^n;\cA_\theta\big)$.

\item[$2.$] The operators $P$ and $P_\rho$ agree.
\end{enumerate}
\end{Proposition}

Combining Proposition~\ref{prop:toroidal.restriction-of-standard-symbol-is-toroidal-symbol} with Proposition~\ref{prop:toroidal.extension-of-toroidal-symbol-is-standard-symbol} we arrive at the following result.

\begin{Proposition}[see \cite{HLP:Part1}] \label{prop:toroidal.standard-and-toroidal-psidos-agree}
For every $m\in \R$, the classes of toroidal and standard \psidos\ of order $m$ agree.
\end{Proposition}

The above extension map $(\rho_k)_{k\in\Z^n}\rightarrow \tilde{\rho}(\xi)$ is a right-inverse of the restriction map $\rho(\xi)\rightarrow (\rho(k))_{k\in \Z^n}$. Applying the latter to a symbol $\rho(\xi)\in \stS^m\big(\R^n;\cA_\theta\big)$ and then applying the extension map produces a symbol $\tilde{\rho}(\xi)\in \stS^m\big(\R^n;\cA_\theta\big)$. Namely, $\tilde{\rho}(\xi)$ is given by~(\ref{eq:toroidal.trho}) with $\rho_k=\rho(k)$, $k\in \Z^n$.

\begin{Proposition}[see~\cite{HLP:Part1}]\label{prop:PsiDOs.tilde-symbol-symbol}
 Let $\rho(\xi)\in \stS^m\big(\R^n;\cA_\theta\big)$, $m\in \R$. Then $P_{\tilde{\rho}}=P_\rho$ and $\tilde{\rho}(\xi)-\rho(\xi)\in \cS\big(\R^n;\cA_\theta\big)$.
\end{Proposition}

The above results allow us to construct a symbol map $\Psi^q(\T^n_\theta)\rightarrow S^q(\cA_\theta)$ for every $q\in \C$ as follows. Set $m=\Re q$. Given any $P\in \Psi^q(\T^n_\theta)$, it follows from~(\ref{eq:toroidal.Prhou-equation}) and Proposition~\ref{prop:toroidal.restriction-of-standard-symbol-is-toroidal-symbol} show $P$ is the \psido\ associated with the toroidal symbol $(\rho_k)_{k\in \Z^n}\in \stS^m\big(\Z^n;\cA_\theta\big)$, where $\rho_k:=P\big(U^k\big)\big(U^{k}\big)^{-1}$, $k\in \Z^n$.
Then Proposition~\ref{prop:toroidal.extension-of-toroidal-symbol-is-standard-symbol} asserts that $P=P_{\rho_{\!{}_{P}}}$, where $\rho_{\!{}_{P}}(\xi)\in \stS^m\big(\Z^n;\cA_\theta\big)$ is the standard symbol associated with $(\rho_k)_{k\in \Z^n}$ by~(\ref{eq:toroidal.trho}). That is,
\begin{gather}
 \rho_{\!{}_{P}}(\xi)= \sum_{k\in \Z^n} \phi(\xi-k) P\big(U^k\big)\big(U^{k}\big)^{-1}, \qquad \xi\in \R^n.
 \label{eq:PsiDOs.symbol-map}
\end{gather}

Let $\rho(\xi)\in S^q\big(\R^n;\cA_\theta\big)$ be such that $P=P_\rho$. Then $\rho(k)=P_\rho(U_k)\big(U^{k}\big)^{-1}=\rho_{\!{}_{P}}(k)$ for all $k\in \Z^n$, and so $\tilde{\rho}(\xi)=\rho_{\!{}_{P}}(\xi)$. It then follows from Proposition~\ref{prop:PsiDOs.tilde-symbol-symbol} that $\rho_{\!{}_{P}}(\xi)-\rho(\xi)\in \cS\big(\R^n;\cA_\theta\big)$, and so $\rho_{\!{}_{P}}(\xi)\in S^q\big(\R^n;\cA_\theta\big)$. Therefore,
we arrive at the following result.

\begin{Proposition} For every $q\in \C$, the formula~\eqref{eq:PsiDOs.symbol-map} defines a linear map $P \rightarrow \rho_{\!{}_{P}}(\xi)$ from $\Psi^q(\T^n_\theta)$ to $S^q\big(\R^n;\cA_\theta\big)$ that is a right inverse of the quantization map $\rho \rightarrow P_\rho$. That is,
\begin{gather*}
P= P_{\rho_{\!{}_{P}}} \qquad \forall\, P\in \Psi^q(\T^n_\theta).
%\label{eq:PsiDOs.symbol-map-PsiDO}
\end{gather*}
\end{Proposition}

\subsection{Smoothing operators}
In what follows we denote by $\cA_\theta'$ the topological dual of $\cA_\theta$ equipped with its strong dual topology. Note that $\cA_\theta$ embeds into $\cA_\theta'$ as a dense subspace (see, e.g., \cite{HLP:Part1}).

\begin{Definition}
 A linear operator $R\colon \cA_\theta \rightarrow \cA_\theta'$ is called \emph{smoothing} when its range is contained in $\cA_\theta$ and it extends to a continuous linear map $R\colon \cA_\theta'\rightarrow \cA_\theta$.
\end{Definition}

In what follows we denote by $\Psi^{-\infty}(\T^n_\theta)$ the class of smoothing operators.

\begin{Proposition}[{\cite[Proposition~6.30]{HLP:Part1}}] \label{prop:PsiDOs.smoothing-operator-characterization}
Let $R\colon \cA_\theta \rightarrow \cA_\theta'$ be a linear map. The following are equivalent.
\begin{enumerate}\itemsep=0pt
\item[$(i)$] $R$ is a smoothing operator.

\item[$(ii)$] $R$ is the toroidal \psido\ associated with some symbol in $\cS\big(\Z^n; \cA_\theta\big)$.

\item[$(iii)$] $R$ is the \psido\ associated with some symbol in $\cS\big(\R^n; \cA_\theta\big)$.
\end{enumerate}
\end{Proposition}

\begin{Remark}
 As $\cS\big(\R^n;\cA_\theta\big)=\bigcap\limits_{q\in \C}S^q\big(\R^n;\cA_\theta\big)$, we see that
 $\Psi^{-\infty}(\T^n_\theta)= \bigcap\limits_{q\in \C}\Psi^q(\T^n_\theta)$.
\end{Remark}

\subsection{Boundedness and trace-class properties}
An extensive account on Sobolev-mapping properties and spectral theoretic properties of \psidos\ is given in~\cite{HLP:Part2}. In this paper we will only need the following two results.

\begin{Proposition}[see~\cite{HLP:Part2}]
 Let $\rho(\xi)\in \stS^m\big(\R^n;\cA_\theta\big)$, $m\leq 0$. Then the operator $P_\rho$ uniquely extends to a bounded operator
 $P_\rho\colon \cH_\theta \rightarrow \cH_\theta$. We obtain a compact operator when $m<0$.
\end{Proposition}

This result allows us to identify $\Psi^q(\T^n_\theta)$ with a subspace of $\cL(\cH_\theta)$ when $\Re q\leq 0$.

\begin{Proposition}[\cite{HLP:Part2}; see also~\cite{LM:GAFA16}]\label{prop:PsiDOs.trace-class}
 Let $P\in \Psi^q(\T^n_\theta)$, $\Re q<-n$. Then $P$ is a trace-class operator, and we have
\begin{gather*}
 \Tr ( P ) = \int_{\R^n} \tau\big[\rho_{\!{}_{P}}(\xi)\big]{\rm d}\xi,
 %\label{eq:PsiDOs.trace-formula-rhoP}
\end{gather*}
where $\rho_{\!{}_{P}}(\xi)$ is the symbol of $P$ given by~\eqref{eq:PsiDOs.symbol-map}.
\end{Proposition}

\begin{Remark}\label{rmk:PsiDOs.trace-formula}
 In general (cf.~\cite{HLP:Part2}), for any symbol $\rho(\xi)\in \stS^{m}\big(\R^n;\cA_\theta\big)$, $m<-n$, the corresponding operator $P_\rho$ is trace-class, and we have
\begin{gather}
 \Tr\big(P_\rho\big) = \sum_{k\in \Z^n} \tau\big[\rho(k)\big]. \label{eq:PsiDOs.trace-formula}
\end{gather}
\end{Remark}

\section{Noncommutative residue and canonical trace}\label{sec:NCR-TR}
In this section, we review the constructions of the noncommutative residue on integer-order \psidos\ and of the canonical trace on non-integer-order \psidos\ on NC tori. We focus more especially on their trace properties. In what follows, given any subset $X\subset \C$, we set
\begin{gather*}
 \Psi^X(\T^n_\theta):= \bigcup_{q\in X} \Psi^{q}(\T^n_\theta), \qquad S^{X} \big(\R^n;\cA_\theta\big):= \bigcup_{q\in Z} S^q\big(\R^n;\cA_\theta\big).
\end{gather*}
For instance, $\Psi^\C(\T^n_\theta)$ consists of all classical \psidos\ on $\T^n_\theta$.

\subsection{Noncommutative residue}
The noncommutative residue is defined on $\Psi^\Z(\T_\theta^n)$. This is a sub-algebra of $\cL(\cA_\theta)$. Thus, by a~trace on $ \Psi^\Z(\T^n_\theta)$ we shall mean any linear form that is annihilated by the commutator subspace,
\begin{gather}
 \big[ \Psi^\Z(\T^n_\theta),\Psi^\Z(\T^n_\theta)\big]:=\op{Span}\big\{[P_1,P_2]; \, P_j\in \Psi^\Z(\T^n_\theta)\big\}. \label{eq:NCR.PsiX}
\end{gather}

\begin{Definition}
 The \emph{noncommutative residue} is the linear functional $\Res \colon \Psi^\Z(\T^n_\theta)\rightarrow \C$ given by
\begin{gather*}
 \Res(P) = \int_{\bS^{n-1}} \tau [ \rho_{-n}(\xi) ] {\rm d}^{n-1}\xi, \qquad P\in \Psi^{\Z}(\cA_\theta),
\end{gather*}
where $\rho_{-n}(\xi)$ is the symbol of degree~$-n$ of $P$.
\end{Definition}

\begin{Remark} If $P\in \Psi^\Z(\T_\theta^n)$, then the homogeneous symbol $\rho_{-n}(\xi)$ is uniquely determined by~$P$ (cf.~Remark~\ref{rem:PsiDOs.symbol-uniqueness}). We make the convention that $\rho_{-n}(\xi)=0$ when $P$ has order~$<-n$.
\end{Remark}

\begin{Remark} It is immediate from the above definition that $\Res(P)=0$ when its symbol of degree~$-n$ is zero. In particular, the noncommutative residue vanishes on the following classes of operators:
\begin{enumerate}\itemsep=0pt
 \item[(i)] \psidos\ of order~$<-n$, including smoothing operators.

 \item[(ii)] Differential operators (since for such operators the symbols are polynomials without any homogeneous component of negative degree).
\end{enumerate}
\end{Remark}

\begin{Remark}
In dimension $n=2$ (resp., $n=4$) we recover the noncommutative residues of~\cite{FW:JPDOA11} (resp., \cite{FK:JNCG15}).
\end{Remark}

\begin{Remark}
 Under the equivalence between our classes of \psidos\ and the classes of toroidal \psidos\ (cf.~Proposition~\ref{prop:toroidal.standard-and-toroidal-psidos-agree}) the above noncommutative residues agrees with the noncommutative residue for toroidal \psidos\ introduced in~\cite{LNP:TAMS16}.
\end{Remark}

\begin{Proposition}[\cite{FK:JNCG15, FW:JPDOA11, LNP:TAMS16}]\label{Prop2} Let $P\in \Psi^q(\T^n_\theta)$ and $Q\in \Psi^{q'}(\cA_\theta)$ be such that $q+q'\in \Z$. Then
\begin{gather*}
 \Res(PQ)= \Res(QP).
\end{gather*}
 In particular, the noncommutative residue $\Res$ is a trace on the algebra $\Psi^{\Z}(\cA_\theta)$.
\end{Proposition}

\begin{Remark}An alternative proof of Proposition~\ref{Prop2} is provided at the end of this section.
\end{Remark}

\subsection{Canonical trace} \label{sec:TR}
The canonical trace is defined on $ \Psi^{\CZ} (\T^n_\theta)$. Note that $ \Psi^{\CZ} (\cA_\theta)$ is not a subspace of $\cL(\cA_\theta)$ and is not closed under the composition of operators. Therefore, we need to explain what is meant by a trace on $ \Psi^{\CZ} (\cA_\theta)$.

\begin{Definition}
 We say that a functional $T\colon \Psi^{\CZ} (\cA_\theta)\rightarrow \C$ is \emph{linear} when
 \begin{gather}
 T(\lambda P)= \lambda T(P) \qquad \forall\, P\in \Psi^{\CZ} (\cA_\theta) \ \forall\, \lambda \in \C,
 \label{eq:TR.linearity1}\\
 T(P_1+P_2)=T(P_1)+T(P_2) \qquad \forall\, P_i\in \Psi^{q_i}(\cA_\theta), \ q_i \not\in \Z, \ q_1-q_2\in \Z.
 \label{eq:TR.linearity2}
 \end{gather}
We say this is a \emph{trace} when it is linear and we have
\begin{gather*}
 T(P_1P_2)=T(P_2P_1) \qquad \forall\, P_i\in \Psi^{q_i}(\cA_\theta), \ q_1+q_2 \not \in \Z.
 %\label{eq:TR.trace-property}
\end{gather*}
\end{Definition}

\begin{Remark}A functional $T\colon \Psi^{\CZ} (\cA_\theta)\rightarrow \C$ is linear if and only if it induces a linear form on each space $\Psi^{q+\Z}(\T^n_\theta)$, $q\in \CZ$. It is trace when we further require it to be annihilated by the commutator subspaces,
\begin{gather*}
 \big[ \Psi^{q_1+\Z}(\T^n_\theta),\Psi^{q_2+\Z}(\T^n_\theta)\big]:=\op{Span}\big\{[P_1,P_2]; \, P_j\in \Psi^{q_j+\Z}(\T^n_\theta)\big\}, \qquad q_1+q_2\in \CZ.
\end{gather*}
 \end{Remark}

Linear traces on non-integer order \psidos\ naturally appear in the analytic extension of the ordinary trace to non-integer order \psidos~ (see~\cite{Gu:AIM93, KV:GDEO, Wo:NCR}). Incidentally, this is a key ingredient in the analytic extension of the zeta functions of elliptic operators and other related spectral functions.

In the setting of noncommutative tori this leads us to the canonical trace $\TR\colon \Psi^{\CZ} (\cA_\theta)\rightarrow \C$ of~\cite{LNP:TAMS16}. Strictly speaking, in~\cite{LNP:TAMS16} the functional is defined on classical toroidal \psidos, but we obtain a functional on $\Psi^{\CZ} (\cA_\theta)$ by using the equality of the class of toroidal \psidos\ with our class of \psidos. Furthermore, the construction in~\cite{LNP:TAMS16} involves the use of extension maps from toroidal symbols to classical symbols. We shall now explain how to define the canonical trace directly on our class of classical \psidos.

In what follows we denote by $ \Psi^{<-n} (\T^n_\theta)$ (resp., $S^{<-n} \big(\R^n;\cA_\theta\big)$) the class of classical \psidos\ (resp., symbols) in~(\ref{eq:NCR.PsiX}) associated with the half-space $X=\{\Re q<-n\}$. By Proposition~\ref{prop:PsiDOs.trace-class} every operator $P\in \Psi^{<-n} (\T^n_\theta)$ is trace-class, and we have
\begin{gather}
 \Tr(P)= \int_{\R^n} \tau\big[\rho_{\!{}_{P}}(\xi)\big] {\rm d}\xi = L [ \tau\circ \rho_{\!{}_{P}} ],
 \label{eq:TR.integral-formula}
\end{gather}
where the symbol $\rho_{\!{}_{P}}(\xi)$ is defined in~(\ref{eq:PsiDOs.symbol-map}) and $L\colon S^{<-n}\big(\R^n\big)\rightarrow \C$ is defined by
\begin{gather}
 L(\rho)= \int_{\R^n} \rho(\xi) {\rm d}\xi . \label{eq:TR.L}
\end{gather}

\begin{Definition} Given any open $\Omega \subset \C$, we shall say that a family of symbols $(\rho(z))_{z\in \Omega}$ in $S^{\C}\big(\R^n;\cA_\theta\big)$ is \emph{holomorphic} when the following properties hold:
 \begin{enumerate}\itemsep=0pt
 \item[(i)] $\rho(z)\in S^{w(z)}\big(\R^n;\cA_\theta\big)$, where $w(z)$ is a holomorphic function on $\Omega$.

 \item[(ii)] The map $\Omega \times \R^n \ni (z,\xi)\rightarrow \rho(z)(\xi)\in \cA_\theta$ is holomorphic with respect to $z$ and $C^\infty$ with respect to $\xi$.

 \item[(iii)] We have an asymptotic expansion,
 \begin{gather}
 \rho(z)(\xi) \sim \sum_{j\geq 0}\rho_{w(z)-j}(z)(\xi), \qquad \rho_{w(z)-j}(z)\in S_{w(z)-j}\big(\R^n;\cA_\theta\big), \label{eq:TR.hol-family-estimates}
 \end{gather}
 in the sense that, for all integers $N\geq 0$, for compact sets $K\subset \Omega$, and multi-orders $\alpha$, $\beta$, there is $C_{NK\alpha\beta}>0$ such that, for all $z\in K$ and $\xi\in \R^n$, $|\xi|\geq 1$, we have
\begin{gather*}
 \bigg\| \delta^\alpha\partial_\xi^\beta\bigg[ \rho(z) -\sum_{j<N}\rho_{w(z)-j}(z)\bigg](\xi) \bigg\| \leq C_{NK\alpha\beta}|\xi|^{\Re w(z)-N-|\beta|}.
\end{gather*}
\end{enumerate}
 \end{Definition}

\begin{Remark}\label{rmk:TR.hol-standard}
 Suppose there is $m\in \R$ such that $\Re w(z)<m$ on $\Omega$. Then the condition (ii) and the estimate~(\ref{eq:TR.hol-family-estimates}) for $N=0$ imply that $(\rho(z))_{z\in \Omega}$ is a holomorphic family in the Fr\'echet space $\stS^{m}\big(\R^n;\cA_\theta\big)$. This can be seen by means of some Cauchy formula argument.
\end{Remark}

\begin{Remark}\label{rmk:TR.A}
 For every symbol $\rho(\xi)\in S^q\big(\R^n;\cA_\theta\big)$ we can find a holomorphic family $(\rho(z))_{z\in \C}$ in $S^{\C}\big(\R^n;\cA_\theta\big)$ such that $\rho(0)=\rho$ and $\ord \rho(z)=z+q$. For instance, we may take
\begin{gather*}
 \rho(z)(\xi):= (1-\chi(\xi) ) |\xi|^z\rho(\xi) + \chi(\xi) \rho(\xi), \qquad \xi\in \R^n, \ z\in \C.
\end{gather*}
Such a family is called a \emph{holomorphic gauging} of~$\rho(\xi)$.
\end{Remark}

In what follows, we shall say that a functional $\varphi\colon S^{\CZ}\big(\R^n\big)\rightarrow \C$ is \emph{analytic} when, for every holomorphic family $(\rho(z))_{z\in \Omega}$ in $S^{\CZ}\big(\R^n\big)$, the function $z\rightarrow \varphi[\rho(z)]$ is analytic on $\Omega$.

\begin{Remark}[analytic continuation principle]\label{rmk:TR.B}
 An analytic functional $\varphi\colon S^{\CZ}\big(\R^n\big)\rightarrow \C$ is uniquely determined by its restriction to symbols of order~$<m$ for any given $m\in \R$. Indeed, suppose that $\varphi$ vanishes on such symbols. Then it must vanish on any holomorphic gauging in the sense of Remark~\ref{rmk:TR.A} by the analytic continuation principle. As explained in Remark~\ref{rmk:TR.A}, any symbol can be embedded into such a family. It then follows that $\varphi$ is identically zero.
\end{Remark}

The key step in the construction of the canonical trace on \psidos\ on (closed) manifolds is the following lemma.

\begin{Lemma}[\cite{Gu:AIM93}; see also~\cite{CM:GAFA95, KV:GDEO, Le:AGAG99, PS:GAFA07}] \label{lem:TR.A}
 The following holds.
\begin{enumerate}\itemsep=0pt
 \item[$(i)$] The functional $L\colon S^{<-n}\big(\R^n\big)\rightarrow \C$ given by~\eqref{eq:TR.L} has a unique analytic extension $\tilde{L}\colon$ $S^{\CZ}\big(\R^n\big)\rightarrow \C$.

 \item[$(ii)$] Let $\rho(\xi)\in S^\Z\big(\R^n\big)$. For every holomorphic gauging $(\rho(z))_{z\in \C}$ for~$\rho$, the composition $\tilde{L}[\rho(z)]$ has at worst a simple pole singularity near $z=0$, and we have
\begin{gather*}
 \Res_{z=0} \tilde{L}\big[\rho(z)\big]=-\int_{\bS^{n-1}} \rho_{-n}(\xi) {\rm d}^{n-1}\xi,
 %\label{eq:TR.residue-formula}
\end{gather*}
where $\rho_{-n}(\xi)$ is the symbol of degree~$-n$ of $\rho(\xi)$.
\end{enumerate}
\end{Lemma}

\begin{Remark}\label{rmk:TR.linearity-tL}
The linearity of $L$ and the analytic continuation principle mentioned in Remark~\ref{rmk:TR.B} ensure us that its analytic extension $\tilde{L}$ is linear in the sense of~(\ref{eq:TR.linearity1})--(\ref{eq:TR.linearity2}).
\end{Remark}

\begin{Definition}
 Given any open set $\Omega\subset \C$, a family $(P(z))_{z\in \Omega}\subset \Psi^{\C}(\cA_\theta)$ is \emph{holomorphic} when there is a holomorphic family of symbols $(\rho(z))_{z\in \Omega}$ such that $P(z)=P_{\rho(z)}$ for all $z\in \Omega$.
\end{Definition}

\begin{Remark}\label{rmk:TR.C}
 Given $P\in \Psi^q(\T^n_\theta)$, let $\rho(\xi)\in S^q\big(\R^n;\cA_\theta\big)$ be such that $P=P_{\rho}$ and let $(\rho(z))_{z\in \C}\subset S^{\C}\big(\R^n;\cA_\theta\big)$ be a holomorphic gauging of $\rho$. Such a family exists by Remark~\ref{rmk:TR.A}. Set $P(z)=P_{\rho(z)}$, $z\in \C$. Then $(P(z))_{z\in \C}$ is a holomorphic family in $\Psi^{\C}(\cA_\theta)$ such that $P(0)=P_{\rho(0)}=P_\rho=P$ and $\ord P(z)=\ord \rho(z)=z+q$. Such a family is called a \emph{holomorphic gauging} of $P$.
 \end{Remark}

\begin{Lemma}\label{lem:TR.B}
 Given any family $(P(z))_{z\in\Omega} \subset \Psi^{\C}(\cA_\theta)$, the following are equivalent:
\begin{enumerate}\itemsep=0pt
 \item[$(i)$] $(P(z))_{z\in\Omega}$ is a holomorphic family in $\Psi^{\C}(\cA_\theta)$.

 \item[$(ii)$] The family of symbols $(\rho_{\!{}_{P(z)}})_{z\in \Omega}\subset S^{\C}\big(\R^n;\cA_\theta\big)$ is holomorphic.
\end{enumerate}
\end{Lemma}
\begin{proof}It is immediate that (ii) implies (i), so we only have to establish the converse. Let $(P(z))_{z\in \Omega}$ be a holomorphic family in $\Psi^{\C}(\cA_\theta)$. Put $P(z)=P_{\rho(z)}$, where $(\rho(z))_{z\in \Omega}$ is a holomorphic family in $S^{\C}\big(\R^n;\cA_\theta\big)$. The order $w(z)$ of $\rho(z)$ is an analytic function on $\Omega$. Thus, for any open set $\Omega'\subsubset \Omega$, there is $m\in \R$ such that $\Re w(z)<m$. As explained in Remark~\ref{rmk:TR.hol-standard} this implies that $(\rho(z))_{z\in \Omega'}$ is a holomorphic family in $\stS^m\big(\R^n;\cA_\theta\big)$.

Given any symbol $\rho(\xi)\in \stS^m\big(\R^n;\cA_\theta\big)$ denote by $\tilde{\rho}(\xi)$ the symbol~(\ref{eq:toroidal.trho}) associated with $\rho_k=\rho(k)$, $k\in \Z^n$.
By~\cite[Lemma~5.21]{LP:Resolvent} this gives rise to a continuous linear map $\rho(\xi) \rightarrow \tilde{\rho}(\xi)-\rho(\xi)$ from $\stS^m\big(\R^n;\cA_\theta\big)$ to $\cS\big(\R^n;\cA_\theta\big)$. It then follows that $(\tilde{\rho}(z)-\rho(z))_{z\in \Omega'}$ is a holomorphic family in $\cS\big(\R^n;\cA_\theta\big)$. As mentioned in \S\ref{subsec:toroidal} we have $\tilde{\rho}(z)(\xi)=\rho_{\!{}_{P_{\rho(z)}}}(\xi)=\rho_{\!{}_{P(z)}}(\xi)$. Therefore, we see that $(\rho_{\!{}_{P(z)}})_{z\in \Omega'}$ and $(\rho(z))_{z\in \Omega'}$ differ by a holomorphic family in $\cS\big(\R^n;\cA_\theta\big)$. We then deduce that $(\rho_{\!{}_{P(z)}})_{z\in \Omega'}$ is a holomorphic family in $S^{\C}\big(\R^n;\cA_\theta\big)$ over every set open set $\Omega'\subsubset \Omega$, and so this is a holomorphic family over all~$\Omega$. This shows that~(i) implies~(ii). The proof is complete.
\end{proof}

\begin{Definition}A functional $T\colon \Psi^{\CZ}(\T^n_\theta)\rightarrow \C$ is \emph{analytic} when, for every holomorphic family $(P(z))_{z\in \Omega}\subset \Psi^{\CZ}(\T^n_\theta)$, the composition $T[P(z)]$ is an analytic function on $\Omega$.
\end{Definition}

\begin{Remark}[analytic continuation principle] \label{rmk:TR.D} As every $P\in \Psi^{\C}(\cA_\theta)$ can be embedded in a~holomorphic gauging, in the same way as in Remark~\ref{rmk:TR.B} we see that any analytic functional on $\Psi^{\C}(\cA_\theta)$ is uniquely determined by its restriction to operators of order~$< m$ for any given $m\in \R$.
\end{Remark}

\begin{Definition}The \emph{canonical trace} $\TR\colon \Psi^{\CZ}\rightarrow \C$ is defined by
\begin{gather}
 \TR(P)= \tilde{L}[ \tau\circ\rho_{\!{}_{P}}], \qquad P\in \Psi^{\CZ}(\T^n_\theta),
 \label{eq:NCR-TR.def-TR}
\end{gather}
where the symbol $\rho_{\!{}_{P}}(\xi)$ is defined as in~(\ref{eq:PsiDOs.symbol-map}).
\end{Definition}

\begin{Remark}
 The linearity of $\tilde{L}$ (cf.~Remark~\ref{rmk:TR.linearity-tL}) and the linearity of the symbol map $P\rightarrow \rho_{\!{}_{P}}$ ensure us that
 $\TR$ satisfies the linearity properties~(\ref{eq:TR.linearity1})--(\ref{eq:TR.linearity2}). We will see later that the canonical trace is indeed a trace (see Proposition~\ref{prop:TR.B} below).
\end{Remark}

\begin{Proposition}\label{prop:TR.A}
The following holds.
\begin{enumerate}\itemsep=0pt
 \item[$1.$] The functional $\TR\colon \Psi^{\CZ}(\T^n_\theta)\rightarrow \C$ is the unique analytic continuation to $\CZ$ of the ordinary trace.

 \item[$2.$] Let $P\in \Psi^{\Z}(\cA_\theta)$. For every holomorphic gauging $(P(z))_{z\in \C}$ of $P$, the composition $\TR[P(z)]\!$ has at worst a simple pole singularity near $z=0$, and we have
 \begin{gather}
 \Res_{z=0} \TR [ P(z) ]=- \Res (P).
 \label{eq:TR.residue}
\end{gather}
\end{enumerate}
 \end{Proposition}
\begin{proof}
If $P\in \Psi^{<-n}(\T^n_\theta)\cap \Psi^{\CZ}(\T^n_\theta)$, then $\rho_{\!{}_{P}}(\xi)\in S^{<-n}\big(\R^n;\cA_\theta\big)$, and so by Lemma~\ref{lem:TR.A} and the integral formula~(\ref{eq:TR.integral-formula}) we have
\begin{gather*}
 \TR(P)= \tilde{L} [ \tau\circ\rho_{\!{}_{P}} ] = L [ \tau\circ\rho_{\!{}_{P}} ] = \Tr(P).
\end{gather*}
Moreover, if $(P(z))_{z\in \Omega}$ is a holomorphic family in $\Psi^{\CZ}(\T^n_\theta)$, then $(\rho_{\!{}_{P(z)}})_{z\in \Omega}$ is a holomorphic family in $S^{\CZ}\big(\R^n;\cA_\theta\big)$. As $\tau$ is a~continuous linear form on $\cA_\theta$ the composition $(\tau\circ \rho_{\!{}_{P(z)}})_{z\in \Omega}$ is a~holomorphic family in $S^{\CZ}\big(\R^n\big)$. The first part of Lemma~\ref{lem:TR.A} then ensures us that $\TR[P(z)]=\tilde{L}[ \tau\circ\rho_{\!{}_{P(z)}}]$ is analytic on~$\CZ$. This shows that $\TR$ is an analytic extension of the ordinary trace to $\Psi^{\CZ}(\T^n_\theta)$. The unique continuation principle mentioned in Remark~\ref{rmk:TR.D} then ensures us that this is the unique such functional on $\Psi^{\CZ}(\T^n_\theta)$.

Finally, given $P\in \Psi^\Z(\T^n_\theta)$, let $(P(z))_{z\in \C}$ be a holomorphic gauging of $P$. It follows from Lemma~\ref{lem:TR.B} and the continuity of $\tau$ that $(\tau[\rho_{\!{}_{P(z)}}])_{z\in \C}$ is a holomorphic gauging of $\tau[\rho_{\!{}_{P}}(\xi)]$ in~$S^{\C}\big(\R^n\big)$. Therefore, by the 2nd part of Lemma~\ref{lem:TR.A} the function $\TR[P(z)]$ has at worst a~simple pole singularity near $z=0$, and we have
\begin{gather*}
 \Res_{z=0} \TR [ P(z) ] = \Res_{z=0} \tilde{L} [ \tau\circ \rho_{\!{}_{P(z)}} ]= - \int_{\bS^{n-1}} \tau [\rho_{-n}(\xi) ] = -\Res(P),
\end{gather*}
where $\rho_{-n}(\xi)$ is the symbol of degree~$-n$ of $P$. The proof is complete.
\end{proof}

\begin{Remark}The right-hand side of~(\ref{eq:NCR-TR.def-TR}) involves the symbol map $P\rightarrow \rho_{\!{}_{P}}$ whose definition~(\ref{eq:PsiDOs.symbol-map}) \emph{a priori} depends on the choice of the function~$\phi$. It follows from the uniqueness contents of the 1st part of Proposition~\ref{prop:TR.A} that the canonical trace is independent of the choice of~$\phi$.
\end{Remark}

\begin{Remark}In~\cite{FGK:MPAG17} the authors attempted to define a canonical trace on NC 2-tori (see~\cite[Definition~2]{FGK:MPAG17}). This definition is ambiguous, since it depends on the choice of symbol, which is not unique. Moreover, on operators of order~$<-n$ this does not agree with the ordinary trace. Namely, for operators $P_\rho$ with $\rho(\xi)\in S^{<-2}\big(\R^2;\cA_\theta\big)$ the definition gives $\TR P_\rho=\int \tau[\rho(\xi)]{\rm d}\xi$, which need not agree with the trace of $P_\rho$ in general (cf.~Remark~\ref{rmk:PsiDOs.trace-formula}). Therefore, this not does not give an analytic extension of the ordinary trace.
\end{Remark}

\begin{Proposition}\label{prop:TR.B}
 The functional $\TR$ is a trace on $\Psi^{\CZ}(\T^n_\theta)$.
\end{Proposition}

\begin{Remark}\label{rmk:Uniqueness.TR-proof}
Proposition~\ref{prop:TR.B} is mentioned in~\cite{LNP:TAMS16}. However, some further arguments are required to complete the proof. It is a bit lengthly to complete the proof by the same approach~\cite{Pa:Private}. For sake of completeness we supply a proof by an alternative approach in the next subsection. The approach is similar to the classical approach of~\cite{Gu:AIM85, KV:GDEO}. This also shows that the noncommutative residue is a trace.
\end{Remark}

\subsection{Proofs of Propositions~\ref{Prop2} and~\ref{prop:TR.B}} The approach relies on the following result.

\begin{Lemma}[\cite{Lee:PhD, LP:Powers}]\label{lem:NCR-TR.composition}
Let $(P(z))_{z\in \C}$ and $(Q(z))_{z\in \C}$ be holomorphic families in $\Psi^\C(\T^n_\theta)$. Then the composition $(P(z)Q(z))_{z\in \C}$ is a holomorphic family in $\Psi^{\C}(\T^n_\theta)$ as well.
\end{Lemma}

We first prove Proposition~\ref{prop:TR.B} since the result will be used to prove Proposition~\ref{Prop2}.

\begin{proof}[Proof of Proposition~\ref{prop:TR.B}]
Let $P_1\in \Psi^{q_1}(\T^n_\theta)$ and $P_2\in \Psi^{q_1}(\T^n_\theta)$ with $q_1+q_2\in \CZ$. Let $(P_1(z))_{z\in \C}$ and $(P_2(z))_{z\in \C}$ be holomorphic gaugings of $P_1$ and $P_2$, respectively. For $z\in \C$ set $R(z):=[P_1(z/2),P_2(z/2)]$. It follows from Lemma~\ref{lem:NCR-TR.composition} that $(R(z))_{z\in \C}$ is a holomorphic gauging of the commutator $R(0)=[P_1,P_2]$, and so $\TR(R(z))$ is an analytic function on $\Omega:=\C\setminus (-q_1-q_2+\Z)$.

Without any loss of generality we may assume $\Re q_1\geq \Re q_2$. Let $z\in \Omega$ be such that $\Re (z+q_1)\leq 0$ and $\Re (z+q_2)<-2n$. This ensures us that $P_1(z/2)$ has non-positive order, and hence is bounded, while $P_2(z)$ and $R(z)$ have order~$<-n$, and hence are trace-class. Thus,
\begin{gather*}
 \TR( R(z))= \Tr( R(z))= \Tr( [P_1(z/2),P_2(z/2)])=0.
\end{gather*}
 The unique analytic continuation principle then implies that $\TR(R(z))=0$ for all $z\in \Omega$. In particular, for $z=0$ we get $\TR([P_1,P_2])=\TR (R(0))=0$. This proves Proposition~\ref{prop:TR.B}.
\end{proof}

\begin{proof}[Proof of Proposition~\ref{Prop2}]
Let $P_1\in \Psi^{q_1}(\T^n_\theta)$ and $P_2\in \Psi^{q_1}(\T^n_\theta)$ with $q_1+q_2\in \Z$. Let $(P_1(z))_{z\in \C}$ and $(P_2(z))_{z\in \C}$ be holomorphic gaugings of $P_1$ and $P_2$, respectively. In the same way as in the proof of Proposition~\ref{prop:TR.B}, the family $([P_1(z/2),P_2(z/2)])_{z\in \C}$ is a holomorphic gauging of $[P_1,P_2]$. Therefore, by using the residue formula~(\ref{eq:TR.residue}) and the trace property of $\TR$ we get
\begin{gather*}
 \Res \big([P_1,P_2]\big) =- \Res_{z=0} \TR\big( [P_1(z/2),P_2(z/2)]\big)=0.
\end{gather*}
This proves Proposition~\ref{Prop2}.
\end{proof}

\section{Pseudodifferential commutators and derivatives}\label{sec:commutators-mod-smoothing}
In this section, we derive a decomposition of \psidos\ as sums of \psido\ commutators modulo smoothing operators.
As with classical \psidos\ on $\R^n$ (see, e.g., \cite{FGLS:JFA96, Gu:AIM93, Le:AGAG99}) this is done by looking at the relationships between derivatives and commutators with elementary differential operators.

\subsection[Sums of $\delta_j$-derivatives]{Sums of $\boldsymbol{\delta_j}$-derivatives}

We have the following relationships between $\delta_j$-derivatives and commutators.

\begin{Lemma}\label{lemA}
 Let $\rho(\xi)\in \stS^m\big(\R^n; \cA_\theta\big)$, $m\in \R$. Then, for $j=1, \ldots, n$, we have
\begin{gather*}
 [\delta_j, P_\rho]=P_{\delta_j\rho}.
\end{gather*}
\end{Lemma}
\begin{proof} Given $j \in \{1,\ldots,n\}$, let $u= \sum\limits_{k\in \Z^n} u_k U^k$ be in $\cA_\theta$. As $\delta_j u = \sum u_k \delta_j\big(U^k\big)= \sum k_j u_k U^k$, we have
 $P_\rho (\delta_j u)= \sum u_k k_j\rho(k) U^k$. We also have
\begin{gather*}
 \delta_j ( P_\rho u ) = \sum_{k\in \Z^n} u_k \delta_j\big( \rho(k) U^k\big) =
 \sum_{k\in \Z^n} u_k (\delta_j \rho(k) )U^k + \sum_{k\in \Z^n} u_k\rho(k) k_jU^k.
\end{gather*}
Thus,
\begin{gather*}
 P_\rho (\delta_j u) -\delta_j ( P_\rho u ) = \sum_{k\in \Z^n} u_k (\delta_j \rho(k) )U^k = P_{\delta_j\rho}u.
\end{gather*}
This shows that $[\delta_j, P_\rho]= P_{\delta_j\rho}$. The proof is complete.
\end{proof}

The above lemma implies that if the symbol of a \psido\ is a sum of $\delta_j$-derivatives, then the operator is a sum of \psido\ commutators. In order to exhibit symbols that are sums of $\delta_j$-derivatives we will need the following lemma.

\begin{Lemma}\label{lemB1}
 Let $\Phi\colon \cA_\theta \rightarrow \cA_\theta$ be a continuous linear map. Then the following holds.
 \begin{enumerate}\itemsep=0pt
 \item[$(i)$] $\Phi$ induces a continuous linear map from $\stS^m\big(\R^n;\cA_\theta\big)$ to itself for every $m\in \R$.

 \item[$(ii)$] $\Phi$ induces a linear map from $S^q\big(\R^n;\cA_\theta\big)$ to itself for every $q\in \C$.
\end{enumerate}
\end{Lemma}
\begin{proof}
 As $\Phi\colon \cA_\theta \rightarrow \cA_\theta$ is a continuous linear map, for every open set $U\subset \R^n$, we get a~(continuous) linear map
$C^\infty(U; \cA_\theta) \ni \rho(\xi) \rightarrow \Phi [ \rho(\xi) ] \in C^\infty(U; \cA_\theta)$. By arguing along similar lines as that of the proof of~\cite[Lemma~4.16(i)]{HLP:Part1} it can be shown that this induces a continuous linear map from $\stS^m\big(\R^n;\cA_\theta\big)$ to itself for every $m\in \R$. This gives~(i).

It remains to prove~(ii). Let $\rho(\xi)\in S^q\big(\R^n;\cA_\theta\big)$, $q\in \C$. We have $\rho(\xi)\sim \sum\limits_{j\geq 0} \rho_{q-j}(\xi)$ with $\rho_{q-j}(\xi) \in S_{q-j}\big(\R^n;\cA_\theta\big)$. Set $m=\Re q$, and let $\chi(\xi)\in C^\infty_c \big(\R^n\big)$ be such that $\chi(\xi)=1$ near $\xi=0$. Then $\rho(\xi)\sim \sum\limits_{j\geq 0} (1-\chi(\xi))\rho_{q-j}(\xi)$, and so $\rho(\xi) -\sum\limits_{j<N} (1-\chi(\xi))\rho_{q-j}(\xi)$ is contained in $\stS^{m-N}\big(\R^n;\cA_\theta\big)$ for every $N\geq 1$. Applying the linear map $\Phi$ and using~(i) we then deduce that
\begin{gather}
 \Phi\big[\rho(\xi)\big] -\sum_{j<N} (1-\chi(\xi))\Phi [\rho_{q-j}(\xi) ]\in \stS^{m-N}\big(\R^n;\cA_\theta\big) \qquad \forall\, N\geq 1.
 \label{eq:delta.Phi-expansion}
\end{gather}
Note that $\Phi[\rho_{q-j}(\xi)]$ is in $C^\infty\big(\R^n\setminus 0; \cA_\theta\big)$ and is homogeneous of degree $q-j$, and so this is a symbol in $S_{q-j}\big(\R^n;\cA_\theta\big)$. Together with~(\ref{eq:delta.Phi-expansion}) this implies that $\Phi[\rho(\xi)]\sim \sum\limits_{j\geq 0} \Phi[\rho_{q-j}(\xi)]$, and hence $\Phi[\rho(\xi)]$ is a symbol in $S^q\big(\R^n;\cA_\theta\big)$. This proves~(ii) and completes the proof.
\end{proof}

\begin{Lemma}\label{lemB}
 Let $\rho(\xi)\in S^q\big(\R^n;\cA_\theta\big)$, $q\in \C$. Then there are symbols $\rho_1(\xi), \ldots, \rho_n(\xi)$ in $S^q\big(\R^n;\cA_\theta\big)$ such that
\begin{gather}
\rho(\xi) = \tau [ \rho(\xi) ] + \delta_1\rho_1(\xi) + \cdots + \delta_n\rho_n(\xi).
\label{eq:delta.decomposition-taurho}
\end{gather}
\end{Lemma}
\begin{proof} Let $\Delta= \delta_1^2 + \cdots + \delta_n^2$ be the flat Laplacian on $\cA_\theta$. This is a selfadjoint unbounded operator on $\cH_\theta$ with domain $\cH_\theta^{(2)}= \big\{ u= \sum u_k U^k\in \cH_\theta;\, \sum |k|^4|u_k|^2 <\infty\big\}$. The unitaries~$U^k$, $k\in \Z^n$, forms an orthonormal eigenbasis for $\Delta$, since we have
 \begin{gather*}
 \Delta\big( U^k\big) = |k|^2U^k \qquad \forall\, k\in \Z^n.
\end{gather*}
Let $\Delta^{-1}\colon \cH_\theta \rightarrow \cH_\theta$ be the partial inverse of $\Delta$. Namely, we have
\begin{gather}
 \Delta^{-1} \big( U^k\big) =
\begin{cases}
 0 & \text{if $k=0$}, \\
 |k|^{-2}U^k & \text{if $k\neq0$} .
\end{cases}
\label{eq:delta.Delta-inverse}
\end{gather}
Note that $\Delta \Delta^{-1}(1)=0$ and $\Delta \Delta^{-1}\big(U^k\big)=|k|^{-2} \Delta\big(U^k\big)=U^k$ if $k\neq 0$. Thus, for all $u = \sum u_k U^k$ in $\cA_\theta$, we have
\begin{gather}
 \Delta \Delta^{-1} u = \sum_{k\in \Z^n} u_k \Delta \Delta^{-1}\big(U^k\big)= \sum_{k\in \Z^n\setminus 0} u_k U^k= u-u_0 = u-\tau(u).
 \label{eq:delta.Delta-Delta-inverse}
\end{gather}
We also observe that~(\ref{eq:delta.Delta-inverse}) implies that $\Delta^{-1}=P_{\rho}$ with $\rho(\xi)=(1-\chi(\xi))|\xi|^{-2}$, where $\chi(\xi)$ is any function in $C^\infty_c\big(\R^n\big)$ such that $\chi(\xi)=1$ near $\xi=0$ and $\chi(\xi)=0$ for $|\xi|\geq 1$. This shows that~$\Delta^{-1}$ is a~\psido\ of order~$-2$, and so it induces a continuous linear map from $\cA_\theta$ to itself.

Bearing all this in mind, let $\rho(\xi)\in S^q\big(\R^n;\cA_\theta\big)$, $q\in \C$. For $j=1, \ldots, n$, set
\begin{gather*}
 \rho_j(\xi)= \delta_j \Delta^{-1} \big(\rho(\xi)\big), \qquad \xi \in \R^n.
\end{gather*}
As $\delta_j \Delta^{-1}\colon \cA_\theta \rightarrow \cA_\theta$ is a continuous linear map, it follows from Lemma~\ref{lemB1} that $\rho_j(\xi) \in S^q\big(\R^n;\cA_\theta\big)$. Moreover, by using~(\ref{eq:delta.Delta-Delta-inverse}) we get
\begin{gather*}
 \sum_{1\leq j \leq n} \delta_j \rho_j(\xi) = \sum_{1\leq j \leq n} \delta_j^2 \Delta^{-1} (\rho(\xi) )= \Delta \Delta^{-1} (\rho(\xi) )= \rho(\xi) - \tau [ \rho(\xi) ].
\end{gather*}
That is, $\rho(\xi) = \tau [ \rho(\xi) ] + \delta_1\rho_1(\xi) + \cdots + \delta_n\rho_n(\xi)$.
The proof is complete.
\end{proof}

\subsection{Sums of finite differences}
The decomposition~(\ref{eq:delta.decomposition-taurho}) reduces the decomposition of \psidos\ as sums of commutators to the case of \psidos\ with scalar symbols.
For classical symbols on $\R^n$ it is natural to attempt to get a decomposition as sums of $\partial_{\xi_j}$-derivatives, since this amounts to taking commutators with coordinates functions $x_1,\ldots,x_n$.

In the noncommutative setting, however, coordinate functions do not exist. Furthermore, even on ordinary tori the coordinate functions $x_j$ do not acts on $C^\infty(\T^n)$ since they are not periodic. Nevertheless, as the following lemma shows, there is a simple relationship between finite differences of symbols and commutators with the unitaries $U_1, \ldots, U_n$.

\begin{Lemma}\label{lemC}
 Let $\rho(\xi)\in \stS^{m}\big(\R^n\big)$, $m \in \R$. Then, for $j=1, \ldots, n$, we have
\begin{gather*}
 P_{\Delta_j \rho} =\big[P_{U_j^{-1}\rho},U_j\big].
\end{gather*}
\end{Lemma}
\begin{proof}
 Let $j\in \{1,\ldots, n\}$. Note that $U_j^{-1}\rho(\xi)\in \stS^m\big(\R^n;\cA_\theta\big)$. Moreover, for all $u= \sum\limits_{k\in \Z^n}u_kU^k$ in $\cA_\theta$, we have{\samepage
\begin{gather*}
 P_{U_j^{-1} \rho}u= \sum_{k\in \Z^n} u_k U_j^{-1} \rho(k)U^k= U_j \bigg( \sum_{k\in \Z^n} u_k \rho(k)U^k\bigg)= U_j^{-1}P_\rho u.
\end{gather*}
This shows that $P_{U_j^{-1}\rho}=U_j^{-1} P_\rho$.}

Given any $k\in \Z^n$ there is $c_{j,k}\in \R$ such that $U_jU^k={\rm e}^{-2{\rm i}\pi c_{j,k}}U^{k+e_j}$ (see, e.g., \cite[equation~(2.4)]{HLP:Part1}). Thus,
\begin{gather*}
 P_\rho\big(U_jU^k\big)={\rm e}^{-2{\rm i}\pi c_{j,k}} P_\rho\big(U^{k+e_j}\big)= {\rm e}^{-2{\rm i}\pi c_{j,k}} \rho(k+e_j)U^{k+e_j}= \rho(k+e_j)U_jU^{k}.
\end{gather*}
This implies that $U_j^{-1}P_{\rho}\big(U_jU^k\big) =\rho(k+e_j)U^k$. It then follows that, for all $u= \sum\limits_{k\in \Z^n}u_kU^k$ in~$\cA_\theta$, we have
\begin{gather*}
 U_j^{-1}\big(P_{\rho}U_j\big)u= \sum_{k\in \Z^n} u_k U_j^{-1}P_{\rho}\big( U-jU^k\big)= \sum_{k\in \Z^n} u_k \rho(k+e_j)U^k=P_{\rho(\cdot+e_j)}u.
\end{gather*}
This shows that $U_j^{-1}P_{\rho}U_j= P_{\rho(\cdot+e_j)}$. Therefore, we have
\begin{gather*}
 \big[P_{U_j^{-1}\rho}, U_j\big]= \big[U_j^{-1}P_{\rho}, U_j\big]=U_j^{-1}P_{\rho}U_j-P_\rho= P_{\rho(\cdot+e_j)}-P_\rho=P_{\Delta_j\rho}.
\end{gather*}
The result is proved.
\end{proof}

The above lemma suggests to seek for a decomposition of symbols as sums of finite differences. The strategy is to use the decomposition as sums of $\partial_{\xi_j}$-derivatives and re-express the $\partial_{\xi_j}$-derivatives as finite differences.

\begin{Lemma}[\cite{FGLS:JFA96, Gu:AIM85}]\label{lemE} Let $\rho(\xi)\in S^q(\R^n)$, $q\in \C$. Suppose that one of the following conditions are satisfied:
\begin{enumerate}\itemsep=0pt
 \item[$(i)$] $q$ is not an integer~$\geq -n$.

 \item[$(ii)$] $q$ is an integer~$\geq -n$ and $\int_{\bS^{n-1}} \rho_{-n}(\xi){\rm d}^{n-1}\xi=0$.
\end{enumerate}
Then there are symbols $\rho_1(\xi), \ldots, \rho_n(\xi)$ in $S^{q+1}\big(\R^n\big)$ such that
\begin{gather*}
 \rho(\xi)= \partial_{\xi_1}\rho(\xi) + \cdots + \partial_{\xi_n}\rho(\xi) \quad \bmod \cS\big(\R^n\big).
\end{gather*}
\end{Lemma}

The following lemma provides us with a relationship between partial differential derivatives and finite differences for classical symbols. Its proof is postponed to Section~\ref{sec:lemD}.

\begin{Lemma}\label{lemD}
 Let $\rho(\xi)\in S^{q}\big(\R^n\big)$, $q\in \C$. Then, for $j=1,\ldots, n$, there is a symbol $\rho_j(\xi) \in S^{q}\big(\R^n\big)$ such that
$\partial_{\xi_j}\rho(\xi) = \Delta_j\rho_j(\xi) \bmod \cS\big(\R^n\big)$.
\end{Lemma}

By combining Lemmas~\ref{lemE} and~\ref{lemD} we arrive at the following description of sums of finite differences modulo $\cS\big(\R^n\big)$.

\begin{Lemma}\label{lemE'} Let $\rho(\xi)\in S^q\big(\R^n\big)$, $q\in \C$, satisfy one of the conditions $(i)$ or $(ii)$ of Lemma~{\rm \ref{lemE}}. Then there are symbols
$\rho_1(\xi), \ldots, \rho_n(\xi)$ in $S^{q+1}\big(\R^n\big)$ such that
\begin{gather*}
 \rho(\xi)=\Delta_1\rho_1(\xi) + \cdots + \Delta_n\rho_n(\xi) \quad \bmod \cS\big(\R^n\big).
\end{gather*}
\end{Lemma}

\subsection[Sums of $\Psi$DOs commutators]{Sums of $\boldsymbol{\Psi}$DOs commutators}
We are in a position to prove the main result of this section.

\begin{Proposition}\label{Prop3}
Let $P\in \Psi^q(\T^n_\theta)$, $q\in \C$. Assume that either $q\not \in \Z$, or $q\in \Z$ and $\Res (P)=0$. Then there are operators $P_1, \ldots, P_n$ in $\Psi^{q+1}(\T^n_\theta)$ and operators $Q_1, \ldots, Q_n$ in $\Psi^{q}(\cA_\theta)$ such that
\begin{gather}
 P= [P_1,U_1]+\cdots + [P_n,U_n] + [\delta_1, Q_1] +\cdots + [\delta_n, Q_n] \quad \bmod \Psi^{-\infty}(\T^n_\theta).
 \label{eq:delta.sum-commutators}
\end{gather}
\end{Proposition}
\begin{proof}
 Let us write $P=P_\rho$ with $\rho(\xi)\in S^{q}\big(\R^n; \cA_\theta\big)$. By Lemma~\ref{lemB} there are symbols $\sigma_1(\xi),\ldots, \sigma_n(\xi)$ in
 $S^{q}\big(\R^n;\cA_\theta\big)$ such that
\begin{gather*}
 \rho(\xi) = \tau [ \rho(\xi) ]+\delta_1\sigma_1(\xi) + \cdots + \delta_n\sigma_n(\xi).
\end{gather*}
 Here $\tau[\rho(\xi)]\in S^q\big(\R^n\big)$. Moreover, by assumption either $q\not \in \Z$, or $q\in \Z$ and $\int_{\bS^{n-1}}\tau[\rho_{-n}(\xi)]{\rm d}^{n-1}\xi\allowbreak =\Res(P)=0$. This implies that
 $\tau[\rho(\xi)]$ satisfies the assumptions of Lemma~\ref{lemE'}. Therefore, there are symbols $\rho_1(\xi),\ldots, \rho_n(\xi)$ in $S^{q+1}\big(\R^n\big)$ such that
 \begin{gather*}
 \tau[\rho(\xi)]= \Delta_1\rho_1(\xi) + \cdots + \Delta_n\rho_n(\xi) \quad \bmod \cS\big(\R^n\big).
\end{gather*}
 Thus,
\begin{gather*}
 \rho(\xi) = \Delta_1\rho_1(\xi) + \cdots + \Delta_n\rho_n(\xi)+\delta_1\sigma_1(\xi) + \cdots + \delta_n\sigma_n(\xi) \quad \bmod \cS\big(\R^n\big).
\end{gather*}
 Combining this with Lemma~\ref{lemA} and Lemma~\ref{lemC} we then obtain
\begin{gather*}
 P=P_\rho= [P_1,U_1]+\cdots + [P_n,U_n] + [\delta_1, Q_1] +\cdots + [\delta_n, Q_n] \quad \bmod \Psi^{-\infty}(\T^n_\theta),
\end{gather*}
where we have set $P_j=P_{U_j\rho}\in \Psi^{q+1}(\T^n_\theta)$ and $Q_j=P_{\sigma_j}\in \Psi^{q}(\cA_\theta)$. The proof is complete.
\end{proof}

As a consequence of Proposition~\ref{Prop3} we can easily recover the uniqueness result of~\cite{LNP:TAMS16} for the noncommutative residue. Recall that a trace on $\Psi^\Z(\T^n_\theta)$ is called \emph{singular} when it is annihilated by smoothing operators.

\begin{Corollary}[\cite{LNP:TAMS16}]\label{cor:Comm-mod-smooth.unique-singular-trace}
 The noncommutative residue is the unique singular trace up to constant multiple on $\Psi^\Z(\T^n_\theta)$.
\end{Corollary}
\begin{proof}
The noncommutative residue is a singular trace. Conversely, let $\varphi$ be a singular trace on $\Psi^\Z(\T^n_\theta)$. Let $P_0\in \Psi^\Z(\T^n_\theta)$ be such that $\Res (P_0)=1$. Let $P\in \Psi^\Z(\T^n)$, and set $\tilde{P}=P-\Res (P)P_0$. Then $\Res (\tilde{P})=\Res(P)-\Res(P)\Res(P_0)=0$, and so it follows from
Proposition~\ref{Prop3} that $P=\Res(P)P_0$ modulo $\big[\Psi^\Z(\T^n_\theta), \Psi^\Z(\T^n_\theta)\big]+\Psi^{-\infty}(\T^n_\theta)$. As $\varphi$ is annihilated by $\big[\Psi^\Z(\T^n_\theta), \Psi^\Z(\T^n_\theta)\big]+\Psi^{-\infty}(\T^n_\theta)$, we deduce that $\varphi(P)=\Res(P)\varphi(P_0)$. That is, $\varphi =c\Res$, with $c:=\varphi(P_0)$. Thus, every singular trace on $\Psi^\Z(\T^n_\theta)$ is a scalar multiple of the noncommutative residue. The proof is complete.
\end{proof}

\begin{Remark}
 It was shown in~\cite{FK:JNCG15, FW:JPDOA11} that in dimension $n=2$ and $n=4$ the noncommutative residue is the unique exotic trace up to constant multiple on $\Psi^\Z(\T^n_\theta)$. In the terminology of~\cite{LNP:TAMS16} a trace on $\Psi^\Z(\T^n_\theta)$ is called \emph{exotic} when it vanishes on \psidos\ of order~$\leq m$ for some given $m\in \Z$.
\end{Remark}

\section{Commutator structure of smoothing operators}\label{sec:commutator-smoothing}
The decomposition~(\ref{eq:delta.sum-commutators}) holds modulo smoothing operators. In this section, we study the commutator structure of smoothing operators.

It is well known that on a closed manifold a smoothing operator is a sum of commutators of smoothing operators if and only if its trace vanishes (see~\cite{Gu:JFA93, Wo:HDR}). We shall now establish the analogue of this result for smoothing operators on $\T_\theta^n$. To reach this end we need the following characterization of smoothing operators.

\begin{Lemma}
 Let $R\in \cL(\cH_\theta)$. The following are equivalent:
\begin{enumerate}\itemsep=0pt
\item[$(i)$] $R$ is a smoothing operator.

\item[$(ii)$] For every $N\geq 0$, there is $C_N>0$ such that
\begin{gather}
\big| \big\langle R\big(U^k\big)| U^{\ell}\big\rangle \big| \leq C_N (1+|k|+|\ell| )^{-N} \qquad \forall\, k,\ell \in \Z^n.
 \label{eq:Smoothing.smoothing-estimates}
\end{gather}
\end{enumerate}
\end{Lemma}
\begin{proof} The proof relies on the following claim.

\begin{claim*}
 A sequence $(a_k)_{k\in \Z^n}\subset \cA_\theta$ is contained in $\cS\big(\Z^n;\cA_\theta\big)$ if and only if, for every $N\geq 0$, there is $C_{N}>0$ such that
\begin{gather}
 \big| \big \langle a_k| U^{\ell}\big \rangle \big| \leq C_N (1+|k|+|\ell| )^{-N} \qquad \forall\, k,\ell \in \Z^n.
 \label{eq:Smoothing.smoothing-estimates-ak}
\end{gather}
\end{claim*}
\begin{proof}[Proof of the claim]
The Fourier series decomposition $u=\sum u_k U^k$ provides us with a topological isomorphism between $\cA_\theta$ and $\cS\big(\Z^n\big)$ (see~\cite{Co:Foliations82, HLP:Part1}). This means that the topology of $\cA_\theta$ is generated by the semi-norms,
\begin{gather*}
 u \longrightarrow \sup_{\ell \in \Z^n} (1+|\ell|)^{-M} \big|\big\langle u| U^{\ell}\big\rangle \big|, \qquad M\geq 0.
\end{gather*}
Therefore, a sequence $(a_k)_{k\in \Z^n}\subset \cA_\theta$ is contained in $\cS\big(\Z^n;\cA_\theta\big)$ if and only if, for all $M,N\geq 0$ there is $C_{MN}>0$ such that
\begin{gather*}
 \sup_{l\in \Z^n} (1+|\ell| )^{-M} \big|\big\langle u|U^{\ell}\big\rangle \big|\leq C_{MN} (1+|k| )^{-N} \qquad \forall\, k\in \Z^n
\end{gather*}
That is,
\begin{gather}
\big|\big\langle u|U^{\ell}\big\rangle \big|\leq C_{MN} (1+|k| )^{-N} (1+|\ell| )^{-M} \qquad \forall\, k,\ell \in \Z^n.
\label{eq:Smoothing.smoothing-estimates-ak-MN}
\end{gather}
The proof of the claim is then completed by observing that the estimates~(\ref{eq:Smoothing.smoothing-estimates-ak-MN}) are equivalent to the estimates~(\ref{eq:Smoothing.smoothing-estimates-ak}) thanks to the inequalities
\[
(1+|k|)^{\frac12}(1+|\ell|)^{\frac12}\leq (1+|k|+|\ell|) \leq (1+|k|)(1+|\ell|).\tag*{\qed}
\]\renewcommand{\qed}{}
\end{proof}

Suppose that $R$ is a smoothing operator. By Proposition~\ref{prop:PsiDOs.smoothing-operator-characterization} this is the toroidal \psido\ associated with a symbol $(\rho_k)_{k\in \Z^n}$ in
$\cS\big(\Z^n;\cA\big)$. We then have $R\big(U^k\big)=\rho_kU^k$ for all $k\in \Z^n$. Note that if a sequence $(a_k)_{k\in \Z^n}$ is in $\cS\big(\Z^n;\cA_\theta\big)$, then
$\big(a_kU^k\big)_{k\in \Z^n}$ is in $\cS\big(\Z^n;\cA\big)$ as well (see~\cite[Lemma~3.29]{HLP:Part1}). Therefore, we see that $\big(R\big(U^k\big)\big)_{k\in \Z^n}\in \cS\big(\Z^n;\cA_\theta\big)$. The claim above then ensures that~$R$ satisfies the estimates~(\ref{eq:Smoothing.smoothing-estimates}).

Conversely, suppose that the estimates~(\ref{eq:Smoothing.smoothing-estimates}) are satisfied. The above claim implies that the sequence $\big(R\big(U^k\big)\big)_{k\in \Z^n}$ is in $\cS\big(\Z^n;\cA_\theta\big)$. Set $\rho_k= R\big(U^k\big)\big(U^k\big)^{-1}$, $k\in \Z^n$. Then $(\rho_k)_{k\in \Z^n}$ is in $\cS\big(\Z^n;\cA_\theta\big)$ as well.
Let $R'$ be the associated toroidal \psido. This is a smoothing operator by Proposition~\ref{prop:PsiDOs.smoothing-operator-characterization}.
Moreover, by using the continuity of~$R$, for all $u=\sum\limits_{k\in \Z^n} u_kU^k$ in $\cA_\theta$, we have
\begin{gather*}
 Ru= \sum_{k\in \Z^n} u_k R\big(U^k\big) = \sum_{k\in \Z^n} u_k \rho_kU^k = R'u.
\end{gather*}
 Thus, $R'=R$, and hence $R$ is a smoothing operator. The proof is complete.
\end{proof}

In the terminology of~\cite[Appendix~A]{Gu:JFA93}, the estimates~(\ref{eq:Smoothing.smoothing-estimates}) mean that the operator $R$ is smoothing with respect to the orthonormal basis $\big\{U^k\big\}_{k\in \Z^n}$. Therefore, by using~\cite[Theorem~A.1]{Gu:JFA93} we arrive at the following statement.

\begin{Proposition}\label{lemF}
For every $R\in \Psi^{-\infty}(\T^n_\theta)$ such that $\Tr(R)=0$ there are operators $R_1,\ldots, R_4$ in $\Psi^{-\infty}(\T^n_\theta)$ such that $R=[R_1,R_2]+[R_3,R_4]$.
\end{Proposition}

Thanks to the above lemma we see that in order to show that every smoothing operators is a~sum of \psido\ commutators we only need to exhibit one such operator that has a non-zero trace. For smoothing operators in~$\R^n$ this is done in~\cite{Po:JAM10} by using the Green function of the Laplacian to construct a smoothing operator whose symbol is a sum of derivatives. We shall proceed somewhat similarly by exhibiting a smoothing operator with non-zero trace whose symbol is a~sum of finite differences.

\begin{Lemma}\label{lemN5}
 Let $\chi(\xi) \in \cS\big(\R^n\big)$ be such that $\check{\chi}(x)=1$ near $x=0$ and $\supp \check{\chi} \subset (-2\pi,2\pi)^n$. Then there are symbols $\rho_1(\xi), \ldots, \rho_{n}(\xi)$ in $S^{-n+2}\big(\R^n\big)$ such that $\chi=\Delta_1\rho_1 +\cdots +\Delta_n\rho_n$.
\end{Lemma}

The proof of Lemma~\ref{lemN5} is postponed to Section~\ref{sec:lemN5}. Assuming this lemma we shall prove the following result.

 \begin{Lemma}\label{lemN}
There exists $R_0\in \Psi^{-\infty}(\T^n_\theta)$ with $\Tr(R_0)=1$ for which there are $P_1,\ldots, P_n$ in $\Psi^{-n+1}(\cA_\theta)$ such that
$ R_0= [P_1,U_1]+ \cdots +[P_n,U_n]$.
\end{Lemma}
\begin{proof} Let $\chi(\xi)\in \cS\big(\R^n\big)$ be such that $\check{\chi}(x)=1$ near $x=0$ and $\supp \check{\chi} \subset (-2\pi, 2\pi)^n$. Set $R_0=(2\pi)^{-n}P_\chi$. Combining the trace formula~(\ref{eq:PsiDOs.trace-formula}) with Poisson's summation formula gives
\begin{gather*}
 \Tr(R_0)= (2\pi)^{-n} \sum_{k\in \Z^n} \chi(k) = \sum_{k\in \Z^n} \check{\chi}(2\pi k) = \check{\chi}(0)=1.
\end{gather*}
 Moreover, by Lemma~\ref{lemN5} there are symbols $\rho_1(\xi), \ldots, \rho_{n}(\xi)$ in $S^{-n+2}\big(\R^n\big)$ such that $\chi=\Delta_1\rho_1 +\cdots +\Delta_n\rho_n$.
Combining this with Lemma~\ref{lemC} we then get
\begin{gather*}
 R_0= (2\pi)^{-n} P_\chi= [U_1,P_1]+\cdots +[U_n,P_n],
\end{gather*}
where we have set $P_j=(2\pi)^{-n}P_{U_j^{-1}\rho_j}\in \Psi^{-n+1}(\T_\theta^n)$. The result is proved.
\end{proof}

We are now in a position to prove the main result of this section.

\begin{Proposition}\label{Prop4}
 For every $R\in \Psi^{-\infty}(\T^n_\theta)$, there are operators $P_1,\ldots, P_n$ in $\Psi^{-n+1}(\cA_\theta)$ and operators $R_1,\ldots, R_4$ in $\Psi^{-\infty}(\T^n_\theta)$ such that
\begin{gather*}
 R = [P_1,U_1]+ \cdots + [P_n,U_n]+ [R_1,R_2]+[R_3,R_4].
\end{gather*}
In particular, every smoothing operator is contained in $\big[\Psi^\Z(\T^n_\theta),\Psi^\Z(\T^n_\theta)\big]$.
\end{Proposition}
\begin{proof}
 Let $R_0\in \Psi^{-\infty}(\T^n_\theta)$ be as in Lemma~\ref{lemN}. Let $R\in \Psi^{-\infty}(\T^n_\theta)$, and set $\tilde{R}=R-\Tr(R)R_0$. Then $\tilde{R}\in \Psi^{-\infty}(\T^n_\theta)$ and, as $\Tr(R_0)=1$, we have $\Tr(\tilde{R})=0$. Therefore, by Lemma~\ref{lemF} there are operators $R_1,\ldots, R_4$ in $\Psi^{-\infty}(\T^n_\theta)$ such that $\tilde{R}=[R_1,R_2]+[R_3,R_4]$. Thus,
\begin{gather*}
 R= \Tr(R)R_0+\tilde{R} = \Tr(R)R_0+ [R_1,R_2]+[R_3,R_4].
\end{gather*}
 We also know by Lemma~\ref{lemN} that there are operators $P_1,\ldots, P_n$ in $\Psi^{-n+1}(\cA_\theta)$ such that $R_0=\sum [P_j,U_j]$. It then follows that we have
\begin{gather*}
 R= \Tr(R)R_0+\tilde{R} = \sum_{1\leq j \leq n} [\tilde{P}_j,U_j] + [R_1,R_2]+[R_3,R_4],
\end{gather*}
 where we have set $\tilde{P}_j=\Tr(R) P_j\in \Psi^{-n+1}(\cA_\theta)$. The proof is complete.
 \end{proof}

\section{Uniqueness theorems}\label{sec:Uniqueness}
In this section, we prove our main uniqueness results for the noncommutative residue and the canonical trace on noncommutative tori.

\subsection[Integer-order $\Psi$DOs]{Integer-order $\boldsymbol{\Psi}$DOs}
For integer-order \psidos\ we have the following decomposition modulo sums of \psido\ commutators.

\begin{Proposition}\label{prop:Uniqueness.NCR-Com} Let $P_0\in \Psi^{-n}(\cA_\theta)$ be such that $\Res (P_0)=1$. Then, for every $P \in \Psi^m(\cA_\theta)$, $m\in \Z$, there are operators $P_1,\ldots, P_n$ in $\Psi^{m'+1}(\cA_\theta)$ with $m'\!=\!\max(m,-n)$, operators $Q_1,\ldots, Q_n$ in $\Psi^{m}(\T^n_\theta)$, and smoothing operators $R_1, \ldots, R_4$ such that
\begin{gather*}
 P= \Res (P) P_0 + \sum_{1\leq j \leq n} [P_j,U_j] + \sum_{1\leq j \leq n} [\delta_j,Q_j]+[R_1,R_2]+[R_3,R_4].
\end{gather*}
In particular, for all $P\in \Psi^\Z(\T^n_\theta)$, we have
\begin{gather}
 P= \Res (P) P_0 \quad \bmod \big[\Psi^\Z(\T^n_\theta), \Psi^\Z(\T^n_\theta)\big].
 \label{eq:Uniqueness.NCR-Com}
\end{gather}
\end{Proposition}
\begin{proof} Set $\tilde{P}=P-\Res(P)P_0$. We have $\Res \big(\tilde{P}\big)=\Res (P)-\Res (P)\Res (P_0)=0$. Moreover, if $m\geq -n$, then $\tilde{P}\in \Psi^{m}(\T^n_\theta)$. If $m<-n$, then $\tilde{P}$ is in $\Psi^{m}(\T^n_\theta)$, since $\Res (P)= 0$, and so $\tilde{P}=P\in \Psi^{m}(\T^n_\theta)$. In any case, it follows from Proposition~\ref{Prop3} that there are operators $P_1,\ldots, P_n$ in $\Psi^{m+1}(\T^n_\theta)$, operators $Q_1,\ldots, Q_n$ in $\Psi^{m}(\T^n_\theta)$, and an operator $R\in \Psi^{-\infty}(\T^n_\theta)$ such that
 \begin{gather*}
 \tilde{P}= [P_1,U_1] +\cdots + [P_n,U_n] +[\delta_1,Q_1]+\cdots + [\delta_n,Q_n]+ R.
\end{gather*}
Moreover, by Proposition~\ref{Prop4} there are operators $P_1',\ldots, P_n'$ in $\Psi^{-n+1}(\cA_\theta)$ and operators $R_1, \ldots,\allowbreak R_4$ in $\Psi^{-\infty}$ such that
\begin{gather*}
 R=[P_1,U_1] +\cdots + [P_n,U_n] +[R_1,R_2]+[R_3,R_4].
\end{gather*}
 As $P=\Res(P) P_0 +\tilde{P}$ we then deduce that
\begin{gather*}
 \Res (P) P_0 + \sum_{1\leq j \leq n} [\tilde{P}_j,U_j] + \sum_{1\leq j \leq n} [\delta_j,Q_j]+[R_1,R_2]+[R_3,R_4],
\end{gather*}
where we have set $\tilde{P}_j=P_j+P_j'$. Note that $\tilde{P}_j\in \Psi^{m'+1}(\cA_\theta)$ with $m'=\max(m,-n)$. The proof is complete.
\end{proof}

It follows from~(\ref{eq:Uniqueness.NCR-Com}) that if $\Res(P)=0$, then $P\in \big[\Psi^\Z(\T^n_\theta), \Psi^\Z(\T^n_\theta)\big]$. Conversely, we know from Proposition~\ref{Prop2} that the noncommutative residues is annihilated by the commutator space $ \big[\Psi^\Z(\T^n_\theta), \Psi^\Z(\T^n_\theta)\big]$. Therefore, we arrive at the following statement.

\begin{Corollary}
 Let $P\in \Psi^{\Z}(\cA_\theta)$. Then $P\in \big[\Psi^\Z(\T^n_\theta), \Psi^\Z(\T^n_\theta)\big]$ if and only if $\Res(P)=0$.
\end{Corollary}

We are now in a position to prove our first uniqueness result.

\begin{Theorem}\label{thm:Uniqueness.NCR}
 Every trace on the algebra $\Psi^{\Z}(\cA_\theta)$ is a constant multiple of the noncommutative residue.
\end{Theorem}
\begin{proof}
 We know by Proposition~\ref{Prop2} that $\Res$ is a trace on the algebra $\Psi^{\Z}(\cA_\theta)$. Let $\varphi$ be another trace. Let $P_0\in \Psi^{-n}(\cA_\theta)$ be such that $\Res(P_0)=1$. As $\varphi$ vanishes on the commutator space $\big[\Psi^\Z(\T^n_\theta), \Psi^\Z(\T^n_\theta)\big]$, by using~(\ref{eq:Uniqueness.NCR-Com}) we see that, for all $P\in \Psi^\Z(\T^n_\theta)$, we have
\begin{gather*}
 \varphi(P)=\varphi [\Res(P)P_0 ]=\varphi(P_0)\Res(P).
\end{gather*}
This shows that $\varphi =\varphi(P_0)\Res$, i.e., $\varphi$ is a scalar multiple of $\Res$. The proof is complete.
\end{proof}

\begin{Remark} The first uniqueness result for the noncommutative residue was established by Wodzicki~\cite{Wo:HDR} for classical \psidos\ on closed manifolds (see also~\cite{Le:AGAG99, LN:JNCG13, LP:PLMS07, Po:JAM10}). The result was extended to Fourier integral operators by Guillemin~\cite{Gu:JFA93} and \psidos\ with log-polyhomogeneous symbols by Lesch~\cite{Le:AGAG99}. It was also extended to Heisenberg \psidos\ in~\cite{Po:JFA07}.
\end{Remark}

Let $\HH^\bt\big(\Psi^\Z(\T^n_\theta)\big)$ be the Hoschchild cohomology of the algebra $\Psi^\Z(\T^n_\theta)$ with coefficients in $\Psi^\Z(\T^n_\theta)^*$ (see, e.g., \cite{Co:NCG}). As $\HH^0\big(\Psi^\Z(\T^n_\theta)\big)$ is precisely the space of traces on $\Psi^\Z(\T^n_\theta)$ we immediately arrive at the following statement

\begin{Corollary}\label{cor:Uniqueness.HH0}
 $\HH^0\big(\Psi^\Z(\T^n_\theta)\big)=\C$.
\end{Corollary}

\subsection[Non-integer order $\Psi$DOs]{Non-integer order $\boldsymbol{\Psi}$DOs}
We have the following version of Proposition~\ref{prop:Uniqueness.NCR-Com} for non-integer order \psidos.

\begin{Proposition}\label{prop:Uniqueness.TR-Com}
 Let $R_0\in \Psi^{-\infty}(\T^n_\theta)$ be such that $\Tr(R_0)=1$. For every $P\in \Psi^{q}(\cA_\theta)$, $q\in \CZ$, there are operators $P_1,\ldots, P_n$ in $\Psi^{q+1}(\T^n_\theta)$, operators $Q_1,\ldots, Q_n$ in $\Psi^{q}(\cA_\theta)$, and smoothing operators $R_1, \ldots, R_4$ such that
\begin{gather}
 P= \TR (P) R_0 + \sum_{1\leq j \leq n} [P_j,U_j] + \sum_{1\leq j \leq n} [\delta_j,Q_j]+[R_1,R_2]+[R_3,R_4].
 \label{eq:Uniqueness.TR-Com}
\end{gather}
In particular, we have
\begin{gather}
 P= \TR(P) R_0 \quad \bmod \big[\Psi^\Z(\T^n_\theta), \Psi^{q+\Z}(\T^n_\theta)\big].
 \label{eq:Uniqueness.NCR-Com-CZ}
\end{gather}
\end{Proposition}
\begin{proof}
 Let $P\in \Psi^q(\T^n_\theta)$, $q\in \CZ$. By Proposition~\ref{Prop3} there are operators $P_1,\ldots, P_n$ in $\Psi^{m+1}(\T_\theta^n)$, operators $Q_1,\ldots, Q_n$ in $\Psi^{m}(\T^n_\theta)$, and an operator $R\in \Psi^{-\infty}(\T^n_\theta)$ such that
 \begin{gather*}
P= [P_1,U_1] +\cdots + [P_n,U_n] +[\delta_1,Q_1]+\cdots + [\delta_n,Q_n]+ R.
\end{gather*}

Set $\tilde{R}=R-\Tr(R)R_0$. Then $\tilde{R}\in \Psi^{-\infty}(\T^n_\theta)$ and $\Tr\big(\tilde{R}\big)=\Tr(R) -\Tr(R)\Tr(R_0)=0$. Therefore, by Proposition~\ref{lemF} there are operators $R_1, \ldots, R_4$ in $\Psi^{-\infty}(\T^n_\theta)$ such that $\tilde{R}= [R_1,R_2]+[R_3,R_4]$, i.e., $R=\Tr(R)R_0+[R_1,R_2]+[R_3,R_4]$. Thus,
\begin{gather}
 P= \Tr (R) R_0 + \sum_{1\leq j \leq n} [P_j,U_j] + \sum_{1\leq j \leq n} [\delta_j,Q_j]+[R_1,R_2]+[R_3,R_4].
\label{eq:Uniqueness.TR-Com-TrR}
\end{gather}

As $\TR$ is a trace functional on $\Psi^{\CZ}(\T^n_\theta)$ that agrees with the ordinary trace on smoothing operators, we get
\begin{gather*}
 \TR(P)= \Tr(R)\TR(R_0) = \Tr(R)\Tr(R_0)=\Tr(R).
\end{gather*}
Combining this with~(\ref{eq:Uniqueness.TR-Com-TrR}) gives the decomposition~(\ref{eq:Uniqueness.TR-Com}). The proof is complete.
\end{proof}

It follows from~(\ref{eq:Uniqueness.NCR-Com-CZ}) that if $\TR(P)=0$, then $P\in \big[\Psi^\Z(\T^n_\theta), \Psi^{q+\Z}(\T^n_\theta)\big]$. Conversely, it follows from Proposition~\ref{Prop2} that the canonical trace is annihilated by $ \big[\Psi^\Z(\T^n_\theta), \Psi^{q+\Z}(\T^n_\theta)\big]$. Therefore, we obtain the following characterization of $ \big[\Psi^\Z(\T^n_\theta), \Psi^{q+\Z}(\T^n_\theta)\big]$.

\begin{Corollary}
 Let $P\in \Psi^{q}(\T^n_\theta)$, $q\in \CZ$. Then $P\in \big[\Psi^\Z(\T^n_\theta), \Psi^{q+\Z}(\T^n_\theta)\big]$ if and only if
 \mbox{$\TR(P)=0$}.
\end{Corollary}

We are now in a position to prove our second uniqueness result.

\begin{Theorem}\label{thm:Uniqueness.TR}
 Every trace on $\Psi^{\CZ}(\T^n_\theta)$ is a scalar multiple of the canonical trace $\TR$.
\end{Theorem}
\begin{proof}
We know from Proposition~\ref{prop:TR.B} that $\TR$ is a trace on $\Psi^{\CZ}(\T^n_\theta)$. Let $T$ be another trace functional. In addition, let $R_0\in \Psi^{-\infty}(\T^n_\theta)$ be such that $\Tr(R_0)=1$, and let $P\in \Psi^{q}(\T^n_\theta)$, $q\in \Z$. As $T$ is annihilated by $\big[\Psi^\Z(\T^n_\theta), \Psi^{q+\Z}(\T^n_\theta)\big]$ it follows from~(\ref{eq:Uniqueness.NCR-Com-CZ}) that we have
\begin{gather}
T(P)=T [\TR(P)R_0 ]=T(R_0)\TR(P).
\label{eq:Uniqueness.T-TR}
\end{gather}
This shows that $T =T(R_0)\TR$, i.e., $T$ is a scalar multiple of $\TR$. The proof is complete.
\end{proof}

\begin{Remark}
The uniqueness of the canonical trace for classical \psidos\ on a closed manifolds was established by Maniccia--Seiler--Schrohe~\cite{MSS:PAMS08} (see also~\cite{LN:JNCG13, Po:JAM10}).
\end{Remark}

We observe that~(\ref{eq:Uniqueness.T-TR}) implies that a trace on $\Psi^{\CZ}(\T^n_\theta)$ is uniquely determined by its restriction to smoothing operators. Therefore, we have the following consequence of the proof of Theorem~\ref{thm:Uniqueness.TR}.

\begin{Corollary}\label{cor:Uniqueness.TR-Tr}
 The canonical trace $\TR$ is the unique trace on $\Psi^{\CZ}(\T^n_\theta)$ that agrees with the ordinary trace on smoothing operators.
\end{Corollary}

\section{Proof of Lemma~\ref{lemD}}\label{sec:lemD}
In this section, we prove Lemma~\ref{lemD}. The proof is a consequence of a few intermediate results.

We record the following version of Peetre's inequality.

\begin{Lemma}[{see, e.g., \cite[Lemma~I.8.2]{AG:AMS07}}]\label{lemD1} Let $m\in \R$. Then, for all $\xi, \eta\in \R^n$, we have
\begin{gather}
 \big( 1+ |\xi+\eta|\big)^m \leq \big( 1+ |\xi|\big)^m \big( 1+ |\eta|\big)^{|m|}.
 \label{eq:delta.Peetre}
\end{gather}
 \end{Lemma}

The Taylor formula allows us to approximate finite differences $\Delta_j f(\xi)= f(\xi+e_j)-f(\xi)$ by Taylor polynomials $\sum\limits_{1\leq \ell \leq N} \frac{1}{\ell !} \partial_{\xi_j}^\ell f(\xi)$. Conversely, partial derivatives $\partial_{\xi_j}f(\xi)$ can be approximated by finite differences. In particular, for one-variable functions we have the following result.

\begin{Lemma}[{\cite[Theorem~3.3.39]{RT:Birkhauser10}}]\label{lemD2} Let $N$ be an integer~$\geq 2$. Then there is $C_N>0$ such that, for all $\varphi \in C^\infty(\R)$ and $t\in \R$, we have
\begin{gather*}
 \bigg| \varphi'(t) - \sum_{1\leq \ell \leq N-1} \frac{(-1)^{\ell-1}}{\ell} \Delta_1^{\ell}\varphi(t)\bigg| \leq C_N \sup_{s\in [0,N-1]} \big| \varphi^{(N)}(t+s)\big|.
\end{gather*}
 \end{Lemma}

By~\cite[Lemma~6.10]{HLP:Part1} if $\rho(\xi)\in \stS^{m}\big(\R^n; \cA_\theta\big)$, then $\Delta^\alpha \rho(\xi)\in \stS^{m-|\alpha|}\big(\R^n; \cA_\theta\big)$ for every multi-order~$\alpha$. Keeping this in mind, we have the following result.

\begin{Lemma}\label{lemD3} Let $\rho(\xi)\in \stS^{m}\big(\R^n\big)$, $m\in \R$. Then, for $j=1, \ldots, n$, we have
\begin{gather*}
 \partial_{\xi_j}\rho(\xi) \sim \sum_{\ell \geq 0} \frac{(-1)^{\ell}}{\ell+1} \Delta_j^{\ell+1} \rho(\xi).
\end{gather*}
\end{Lemma}
\begin{proof}
 Let $j\in\{1,\ldots, n\}$ and $N\geq 2$. Given any $\xi\in \R^n$, by applying Lemma~\ref{lemD2} to $\varphi(t)=\rho(\xi+te_j)$ we get
\begin{gather}
 \bigg| \partial_{\xi_j} \rho(\xi) - \sum_{1\leq \ell \leq N-1} \frac{(-1)^{\ell-1}}{\ell}\Delta_j^{\ell} \rho(\xi) \bigg| \leq C_N \sup_{0\leq t\leq N-1} \big| \partial_{\xi_j}^N \rho(\xi+te_j)\big|.
 \label{eq:delta.Taylor-partialxi}
\end{gather}
As $\rho(\xi)\in \stS^{m}\big(\R^n\big)$ there is $C>0$ such that $| \partial_{\xi_j}^N \rho(\xi)|\leq C(1+|\xi|)^{m-N}$ for all $\xi \in \R^n$. Together with Peetre's inequality~(\ref{eq:delta.Peetre}) this implies that, for all $\xi\in \R^n$ and $t\in [0,N-1]$, we have
\begin{align*}
\big| \partial_{\xi_j}^N \rho(\xi+te_j)\big| &\leq C(1+|\xi+t_j|)^{m-N}\\
& \leq C(1+|\xi|)^{m-N} (1+t)^{|m-N|}\\
 & \leq 2^{|m-N|} C(1+|\xi|)^{m-N}.
\end{align*}
Combining this with~(\ref{eq:delta.Taylor-partialxi}) we see that there is $C_{N0}>0$ such that, for all $\xi\in \R^n$, we have
\begin{gather*}
 \bigg| \partial_{\xi_j} \rho(\xi) - \sum_{1\leq \ell \leq N-1} \frac{(-1)^{\ell-1}}{\ell}\Delta_j^{\ell} \rho(\xi) \bigg| \leq C_{N0} C(1+|\xi|)^{m-N}.
\end{gather*}

Likewise, given any multi-order $\alpha$, as $\partial_{\xi}^\alpha \rho(\xi)\in \stS^{m-|\alpha|}\big(\R^n\big)$, there is $C_{N\alpha}>0$ such that, for all $\xi\in \R^n$, we have
\begin{align*}
 \bigg| \partial_\xi^\alpha\bigg(\partial_{\xi_j} \rho - \sum_{1\leq \ell \leq N-1} \frac{(-1)^{\ell-1}}{\ell}\Delta_j^{\ell} \rho\bigg)(\xi) \bigg|
 &= \bigg| \partial_{\xi_j} \partial_\xi^\alpha\rho(\xi) - \sum_{1\leq \ell \leq N-1} \frac{(-1)^{\ell-1}}{\ell}\Delta_j^{\ell} \partial_\xi^\alpha\rho(\xi) \bigg| \\
 & \leq C_{N0} C(1+|\xi|)^{m-|\alpha|-N}.
\end{align*}
This shows that
\begin{gather*}
 \partial_{\xi_j} \rho(\xi) - \sum_{1\leq \ell \leq N-1} \frac{(-1)^{\ell-1}}{\ell}\Delta_j^{\ell} \rho(\xi) \in \stS^{m-N}\big(\R^n\big) \qquad \forall\, N\geq 2.
\end{gather*}
It then follows that $ \partial_{\xi_j} \rho(\xi) \sim \sum\limits_{\ell \geq 1 } \frac{(-1)^{\ell-1}}{\ell}\Delta_j^{\ell} \rho(\xi)$. This proves the result.
\end{proof}

As mentioned above, if $\rho(\xi)\in \stS^{m}\big(\R^n; \cA_\theta\big)$, then $\Delta^\alpha \rho(\xi)\in \stS^{m-|\alpha|}\big(\R^n; \cA_\theta\big)$ for every multi-order~$\alpha$. For classical symbols we further have the following result.

\begin{Lemma}\label{lemD4}
 Let $\rho(\xi)\in S^q\big(\R^n; \cA_\theta\big)$, $q\in \C$. Then $\Delta^\alpha \rho(\xi)\in S^{q-|\alpha|}\big(\R^n; \cA_\theta\big)$ for every $\alpha\in \N_0^n$.
\end{Lemma}
\begin{proof}
 It is enough to prove the result when $|\alpha|=1$. Let $j\in\{1,\ldots, n\}$ and $N\geq 1$. Then by Taylor's formula for differential maps with values in locally convex spaces (see, e.g., \cite[Proposition~C.15]{HLP:Part1}), for any $\xi\in \R^n$, we have
\begin{align*}
 \Delta_j \rho(\xi) - \sum_{1\leq \ell \leq N} \frac{1}{\ell!} \partial_{\xi_j}^{\ell}\rho(\xi) & = \rho(x+e_j) - \sum_{0\leq l \leq N} \frac{1}{\ell!} \partial_{\xi_j}^{\ell}\rho(\xi)\\
 &= \frac{1}{N!} \int_0^1 (1-t)^N \partial_{\xi_j}^{N+1}\rho(\xi+te_j){\rm d}t.
\end{align*}
Thus,
\begin{align}
 \bigg\| \Delta_j \rho(\xi) - \sum_{1\leq \ell \leq N} \frac{1}{\ell!} \partial_{\xi_j}^{\ell}\rho(\xi)\bigg\| & \leq
 \frac{1}{N!} \int_0^1 (1-t)^N \big\|\partial_{\xi_j}^{N+1}\rho(\xi+te_j)\big\|{\rm d}t \nonumber\\
 & \leq \frac{1}{(N+1)!} \sup_{0\leq t \leq 1} \big\|\partial_{\xi_j}^{N+1}\rho(\xi+te_j)\big\|.
 \label{eq:delta.Taylor-Deltaj}
\end{align}

Set $m=\Re q$. As $\rho(\xi)\in \stS^m\big(\R^n; \cA_\theta\big)$. There is $C>0$ such that
\begin{gather*}
 \big\|\partial_{\xi_j}^{N+1}\rho(\xi)\big\| \leq C\big(1+|\xi|)^{m-N-1} \qquad \forall\, \xi \in \R^n.
\end{gather*}
Together with Peetre's inequality~(\ref{eq:delta.Peetre}) this implies that, for all $\xi \in \R^n$ and $t\in [0,1]$, we have
\begin{align*}
 \big\|\partial_{\xi_j}^{N+1}\rho(\xi+te_j)\big\| & \leq C\big(1+|\xi+te_j|)^{m-N-1}\\
 & \leq C(1+|\xi|)^{m-N-1} (1+t)^{|m-N-1|} \\
 & \leq 2^{|m-N-1|} C(1+|\xi|)^{m-N-1}.
\end{align*}
Combining this with~(\ref{eq:delta.Taylor-Deltaj}) then shows there is $C_N>0$ such that, for all $\xi \in \R^n$, we have
\begin{gather*}
 \bigg\| \Delta_j \rho(\xi) - \sum_{1\leq \ell \leq N} \frac{1}{\ell!} \partial_{\xi_j}^{\ell}\rho(\xi)\bigg\| \leq C_N (1+|\xi|)^{m-N-1}.
\end{gather*}

Likewise, given any multi-orders $\alpha$ and $\beta$, as $\delta^\alpha \partial_\xi^\beta \rho(\xi)\in S^{q-|\beta|}\big(\R^n;\cA_\theta\big)$, there is $C_{N\alpha \beta}>0$ such that, for all $\xi \in \R^n$, we have
\begin{align*}
 \bigg\| \delta^\alpha \partial_\beta\bigg(\Delta_j \rho - \sum_{1\leq \ell \leq N} \frac{1}{\ell!} \partial_{\xi_j}^{\ell}\rho\bigg)(\xi)\bigg\|
 & = \bigg\| \Delta_j \delta^\alpha \partial_\beta\rho(\xi) - \sum_{1\leq \ell \leq N} \frac{1}{\ell!} \partial_{\xi_j}^{\ell}\delta^\alpha \partial_\beta\rho(\xi)\bigg\| \\
 & \leq C_{N\alpha \beta} (1+|\xi|)^{m-|\beta|-N-1}.
\end{align*}
This shows that
\begin{gather*}
 \Delta_j \rho(\xi) - \sum_{1\leq \ell \leq N} \frac{1}{\ell!} \partial_{\xi_j}^{\ell}\rho(\xi) \in \stS^{m-N-1}\big(\R^n;\cA_\theta\big) \qquad \forall\, N\geq 1.
\end{gather*}
It then follows that $ \Delta_j \rho(\xi) \sim \sum\limits_{\ell \geq 1} \frac{1}{\ell!} \partial_{\xi_j}^{\ell}\rho(\xi)$. As $\partial_{\xi_j}^{\ell}\rho(\xi)\in S^{q-\ell}\big(\R^n;\cA_\theta\big)$ for every $\ell \geq 1$, we then deduce that $ \Delta_j \rho(\xi) \in S^{q-1}\big(\R^n;\cA_\theta\big)$. This proves the result when $|\alpha|=1$. The proof is complete.
\end{proof}

We are now in a position to prove Lemma~\ref{lemD}.

\begin{proof}[Proof of Lemma~\ref{lemD}]
 Let $\rho(\xi)\in S^{q}\big(\R^n\big)$, $q\in \C$, and $j \in \{1,\ldots, n\}$. We know by Lemma~\ref{lemD3} that $\partial_{\xi_j}\rho(\xi) \!\sim\! \sum\limits_{\ell \geq 0}\! \frac{(-1)^{\ell}}{\ell+1}\! \Delta_j^{\ell+1} \rho(\xi).\!$ Moreover, it follows from Lemma~\ref{lemD4} that $\Delta_j^{\ell} \rho(\xi)\!\in\! S^{q-\ell}\!\big(\R^n; \cA_\theta\big).\!$ By the Borel lemma for scalar-valued symbols (see, e.g., \cite[Proposition~I.2.3]{AG:AMS07}) there is $\rho_j(\xi) \in S^q\big(\R^n\big)$ such that $\rho_j(\xi) \sim \sum\limits_{\ell \geq 0}\frac{(-1)^{\ell}}{\ell+1} \Delta_j^{\ell} \rho(\xi)$.

Set $m=\Re q$. The fact that $\rho_j(\xi) \sim \sum\limits_{\ell \geq 0}\frac{(-1)^{\ell}}{\ell} \Delta_j^{\ell+1} \rho(\xi)$ means that, for every $N\geq 1$, we have
\begin{gather*}
 \rho_j(\xi) - \sum_{\ell <N} \frac{(-1)^{\ell}}{\ell+1} \Delta_j^{\ell} \rho(\xi) \in \stS^{m-N}\big(\R^n\big).
\end{gather*}
 As $\Delta_j$ maps $\stS^{m-N}\big(\R^n\big)$ to $ \stS^{m-N-1}\big(\R^n\big)$, we deduce that
\begin{gather*}
 \Delta_j \rho_j(\xi) - \sum_{\ell <N} \frac{(-1)^{\ell}}{\ell+1} \Delta_j^{\ell+1} \rho(\xi) \in \stS^{m-N-1}\big(\R^n\big)\qquad \forall\, N \geq 1.
\end{gather*}
This means that $\Delta_j \rho_j(\xi)\sim \sum\limits_{\ell \geq 0} \frac{(-1)^{\ell}}{\ell+1} \Delta_j^{\ell+1} \rho(\xi)$. As $\partial_{\xi_j}\rho(\xi) \sim \sum\limits_{\ell \geq 0} \frac{(-1)^{\ell}}{\ell+1} \Delta_j^{\ell+1} \rho(\xi)$, we then deduce that $\partial_{\xi_j}\rho(\xi)-\Delta_j \rho(\xi) \in \cS\big(\R^n\big)$. This proves the result.
\end{proof}

\section{Proof of Lemma~\ref{lemN5}}\label{sec:lemN5}
In this section, we prove Lemma~\ref{lemN5}. The proof is based on the characterization of the Fourier transforms of symbols on~$\R^n$.

In what follows we say that a continuous function $K(x)$ on $\R^n\setminus 0$ has \emph{bounded support} when its support is a bounded subset of $\R$. Equivalently, there is $R>0$ such that $K(x)=0$ for $|x|>R$. In particular, if $K(x)$ is a such a function and is integrable near $x=0$, then $K(x)$ defines a~compactly supported distribution on~$\R^n$, and so it has a smooth Fourier trans\-form~$\hat{K}(\xi)$.

\begin{Lemma}\label{lemN3} Let $K(x) \in C^\infty\big(\R^n\setminus 0\big)$ have bounded support. Assume that near $x=0$ there are analytic functions $f(x)$ and $g(x)$ such that
\begin{gather}
 K(x)=\frac{f(x)}{g(x)}, \qquad f(x)=\op{O}\big(|x|^a\big), \qquad g(x) \sim |x|^{2b},
 \label{eq:lemN5.condition-N3a}
\end{gather}
where $a$ and $b$ are non-negative integers such that $a-2b>-n$. Then $\hat{K}(\xi) \in \stS^{-m+1}\big(\R^n\big)$, with $m=n+a-2b$.
\end{Lemma}
\begin{proof}
Set $m=n+a-2b \in \N$. The analyticity of $f(x)$ and $g(x)$ and~(\ref{eq:lemN5.condition-N3a}) ensure us that, for every multi-order $\alpha$, we have
\begin{gather}
 \partial_x^\alpha K(x) = \op{O}\big(|x|^{-n+m-|\alpha|}\big) \qquad \text{near $x=0$}.
 \label{eq:lemN5.condition-N3}
\end{gather}
In particular, for $\alpha=0$ we see that $K(x)$ is integrable near $x=0$. As $K(x)$ is smooth outside the origin and has bounded support, we see that $K(x)\in L^1\big(\R^n\big)$, and so the Fourier transform $\hat{K}(\xi)$ is bounded on $\R^n$. Furthermore, the boundedness of the support of~$K$ implies that $\hat{K}(\xi)$ is smooth on~$\R^n$.

More generally, the asymptotic~(\ref{eq:lemN5.condition-N3}) ensures us that, for every multi-order $\alpha$ with $|\alpha|\leq m-1$, the partial derivative $\partial^\alpha_x K(x)$ is integrable, and hence has a bounded Fourier transform. Thus, there is $C_\alpha>0$ such that, for all $\xi \in \R^n$, we have
\begin{gather*}
 \big| \xi^\alpha \hat{K}(\xi)\big|= \big| \big(\widehat{\partial_x^\alpha K}\big)(\xi)\big| \leq C_\alpha.
\end{gather*}
As $(1+|\xi|+\cdots + |\xi_n|)^{m-1}$ is a universal linear combination of monomials $|\xi_1|^{\alpha_1} \cdots|\xi_n|^{\alpha_n}= |\xi^\alpha|$ with $|\alpha|\leq m-1$ we deduce there is $C>0$ such that, for all $\xi\in \R^n$, we have
\begin{gather*}
 \big| \hat{K}(\xi)\big| \leq C \big( 1+|\xi_1|+ \cdots + |\xi_n|\big)^{-m+1} \leq C ( 1+ |\xi| )^{-m+1}.
\end{gather*}

 Let $\beta\in \N_0^n$. We have $D_\xi^\beta \hat{K}(\xi)=\big(\widehat{x^\beta K}\big)(\xi)$. Note that $x^\beta K(x)$ is a $C^\infty$ function on $\R^n\setminus 0$ with bounded support.
 Moreover, the asymptotic~(\ref{eq:lemN5.condition-N3}) for $K(x)$ implies that $x^\beta K(x)$ satisfies the corresponding asymptotic for $m+|\beta|$. Therefore, by the first part of the proof there is $C_\beta>0$ such that, for all $\xi\in \R^n$, we have
 \begin{gather*}
\big| D_\xi^\beta \hat{K}(\xi)|= \big| \big(\widehat{x^\beta K}\big)(\xi)\big| \leq C_\beta ( 1+ |\xi| )^{-m-|\beta|+1}.
\end{gather*}
This shows that $\hat{K}(\xi)$ is contained in $\stS^{-m+1}\big(\R^n\big)$. The proof is complete.
\end{proof}

 \begin{Lemma}\label{lemN4}
 Let $K(x)\in C^\infty\big(\R^n\setminus 0\big)$ be such that there is $f(x)\in C_c^\infty\big(\R^n\big)$ such that
 \begin{gather*}
 K(x)= |x|^{-2b}f(x), \qquad f(x)=\op{O}\big(|x|^a\big) \quad \text{near $x=0$},
\end{gather*}
where $a$ and $b$ are non-negative integers such that $a-2b>-n$. Then $\hat{K}(\xi) \in S^{-n-a+2b}\big(\R^n\big)$.
\end{Lemma}
\begin{proof}
As $f(x)$ is $C^\infty$ and is $\op{O}\big(|x|^a\big)$ near $x=0$, by Taylor's formula, for every $N> a$, we have
\begin{gather*}
 f(x) = \sum_{a\leq |\alpha| <N}a_\alpha x^\alpha + \sum_{|\alpha|<}x^\alpha r_{N\alpha}(x),
\end{gather*}
where $a_\alpha:= \frac{1}{\alpha !} \partial_x^\alpha f(0)$ and $r_{N\alpha} (x)\in C^\infty\big(\R^n\big)$. Thus, for all $x\in \R^n\setminus 0$, we have
\begin{gather*}
 K(x) = |x|^{-2b}f(x)= \sum_{a\leq |\alpha| <N}a_\alpha x^\alpha |x|^{-2b} + \sum_{|\alpha|<} x^\alpha |x|^{-2b}r_{N\alpha}(x).
\end{gather*}
Note that if $N>2b$ and $|\alpha|=N$, then $x^\alpha |x|^{-2b}r_{N\alpha}(x)$ is $C^{N-2b-1}$. Thus,
\begin{gather*}
 K(x)- \sum_{a\leq |\alpha| <N}a_\alpha x^\alpha |x|^{-2b} \in C^{N-2b-1}\big(\R^n\big) \qquad \forall\, N>2b.
\end{gather*}
Here $K(x)$ is smooth on $\R^n\setminus 0$ and has bounded support. Moreover, $x^\alpha |x|^{-2b}$ is smooth on $\R^n\setminus 0$ and homogeneous of degree~$|\alpha|-2b\geq a-2b>-n$ for $|\alpha|\geq a$. Therefore, it follows from the characterization of the Fourier transforms of classical symbols of negative integer order (see~\cite[Proposition~7.3.2]{Ta:PDE2}) that $\hat{K}(\xi)\in S^{-(n+a-2b)}\big(\R^n\big)$. The proof is complete.
\end{proof}

We are now in a position to prove Lemma~\ref{lemN5}

\begin{proof}[Proof of Lemma~\ref{lemN5}]
Let $\rho(\xi)\in \cS\big(\R^n\big)$. For $j=1, \ldots, n$, we have
\begin{align*}
 (\Delta_j \rho)^\vee(x) & = \int {\rm e}^{{\rm i}x\cdot \xi} ( \rho(\xi+e_j) - \rho(\xi) ) \dbar\xi,\\
 & = \int {\rm e}^{{\rm i}x\cdot (\xi-e_j)} \rho(\xi) \dbar \xi - \int {\rm e}^{{\rm i}x\cdot \xi} \rho(\xi) \dbar \xi\\
 & = \big({\rm e}^{-x_j}-1\big) \int {\rm e}^{{\rm i}x\cdot \xi} \rho(\xi) \dbar \xi,\\
 & = \big({\rm e}^{-x_j}-1\big) \check{\rho}(x).
\end{align*}
 That is, we have
 \begin{gather}
 \Delta_j \rho(\xi) = \big[\big({\rm e}^{-{\rm i}x_j}-1\big) \check{\rho} \big]^\wedge(\xi) .
 \label{eq:lemN5.Deltaj-Fourier}
\end{gather}
This formula extends by continuity to all tempered distributions $\rho(\xi)\in \cS'\big(\R^n\big)$, where the extension of the operator $\Delta_j$ to $\cS'\big(\R^n\big)$ is given by
\begin{gather*}
 \acou{\Delta_j \rho}{u}=-\acou{\rho}{\overline{\Delta}_ju}, \qquad \rho \in \cS'\big(\R^n\big), \ u \in \cS\big(\R^n\big).
\end{gather*}
Here $\overline{\Delta}_j$ is the $j$-th backward finite difference~(\ref{eq:PsiDOs.finite-differences-functions}) on $\cS\big(\R^n\big)$.

Let $\nu\colon \R^n\rightarrow \R$ be the function defined by
\begin{gather*}
 \nu(x)= \sum_{1\leq j \leq n} \big| {\rm e}^{{\rm i}x_j}-1\big|^2, \qquad x \in \R^n.
\end{gather*}
 In particular, we have
 \begin{gather*}
 \nu(x)=0 \Longleftrightarrow \sum_{1\leq j \leq n} \big| {\rm e}^{{\rm i}x_j}-1\big|^2 =0 \Longleftrightarrow {\rm e}^{{\rm i}x_1}= \cdots = {\rm e}^{{\rm i}x_n} = 1 \Longleftrightarrow
 x\in 2\pi \Z^n.
\end{gather*}
We also have
\begin{gather*}
 \nu(x)= \sum_{1\leq j \leq n} \big({\rm e}^{-{\rm i}x_j}-1\big)\big({\rm e}^{{\rm i}x_j}-1\big) = \sum_{1\leq j \leq n} 2(1-\cos (x_j)).
\end{gather*}
Thus, $\nu(x)$ is an analytic function on $\R^n$. Moreover, as $2(1-\cos (t))=t^2 +\op{O}\big(t^4\big)$ near $t=0$, we see that, near $x=0$, we have
\begin{gather}
 \nu(x)= x_1^2 + \cdots + x_n^2 + \op{O}\big(x_1^4 + \cdots + x_n^4\big)= |x|^2 +\op{O}\big(|x|^4\big).
 \label{eq:lemN5.asymptotic-nu}
\end{gather}
In particular, this shows that $\nu(x)\sim |x|^2$ near $x=0$.

Bearing this in mind, let $\chi(\xi)\in \cS^\infty\big(\R^n\big)$ be such that $\check\chi(x)=1$ near $x=0$ and $\supp \check{\chi}\subset (-2\pi, 2\pi)^n$. For $j=1,\ldots, n$, define
\begin{gather*}
 K_j(x)=\frac{{\rm e}^{{\rm i}x_j}-1}{\nu(x)}\check{\chi}(x), \qquad x\neq 0.
\end{gather*}
As $\supp \check{\chi}\subset (-2\pi, 2\pi)^n$ and the only zero of $\nu(x)$ on $(-2\pi, 2\pi)^n$ is the origin, this defines a smooth function with bounded support on
$\R^n\setminus 0$. Moreover, as $\check{\chi}(x)=1$ near $x=0$, we see that $K_j(x)=\big({\rm e}^{{\rm i}x_j}-1\big)\nu(x)^{-1}$ near $x=0$. Here ${\rm e}^{-{\rm i}x_j}-1$ is analytic and is $\op{O}(|x|)$ near $x=0$. As mentioned above, $\nu(x)$ is analytic function such that $\nu(x)\sim |x|^2$ near $x=0$. Therefore, we see that $K_j(x)$ satisfies the assumptions of Lemma~\ref{lemN3} with $a=b=1$, and so $\rho_j(\xi):=\widehat{K}_j(\xi)$ is a symbol in $\stS^{-n+2}(\R^n)$.

Moreover, by using~(\ref{eq:lemN5.Deltaj-Fourier}) we see that $(\Delta_j \rho_j)^\vee(x)$ is equal to
\begin{gather*}
 \big({\rm e}^{-{\rm i}x_j}-1\big)\check{\rho}_j(x)= \big({\rm e}^{-{\rm i}x_j}-1\big) \frac{\big({\rm e}^{{\rm i}x_j}-1\big)}{\nu(x)}\check{\chi}(x) = \big|{\rm e}^{{\rm i}x_j}-1\big|^2\nu(x)^{-1}\check{\chi}(x).
\end{gather*}
Thus,
\begin{gather*}
 \sum_{1\leq j \leq n} \big(\Delta_j \rho_j\big)^\vee(x)= \check{\chi}(x)\nu(x)^{-1}\sum_{1\leq j \leq n}\big|{\rm e}^{{\rm i}x_j}-1\big|^2=\check{\chi}(x)\nu(x)^{-1}\nu(x)=\check{\chi}(x).
\end{gather*}
Taking Fourier transforms then shows that $\chi(\xi)=\Delta_1\rho_1(\xi)+\cdots+ \Delta_n\rho_n(\xi)$.

In order to complete the proof it remains to show that $\rho_j(\xi)\in S^{-n+1}\big(\R^n\big)$ for $j=1,\ldots, n$. Set $\tilde{\nu}(x)=\nu(x)-|x|^2$. Then $\tilde{\nu}(x)$ is analytic function on $\R^n$. By~(\ref{eq:lemN5.asymptotic-nu}) it is $\op{O}\big(|x|^4\big)$ near $x=0$. Given any $N\geq 1$ and $x\in (-2\pi, 2\pi)^n\setminus 0$, we have
\begin{align*}
 \frac{1}{\nu(x)}= \frac{|x|^2}{1-|x|^{-2}\tilde{\nu}(x)}& = \sum_{\ell<N} |x|^{-2\ell-2}\tilde{\nu}(x)^{\ell} + \frac{|x|^{-2N-2}\tilde{\nu}(x)^N}{1-|x|^{-2}\tilde{\nu}(x)}\\
 & = \sum_{\ell<N} \frac{\tilde{\nu}(x)^{\ell}}{|x|^{2\ell+2}} + \frac{\tilde{\nu}(x)^N}{|x|^{2N}\nu(x)}.
\end{align*}
Thus, for all $x\in \R^n\setminus 0$, we have
\begin{gather}
 K_j(x)= \frac{{\rm e}^{{\rm i}x_j}-1}{\nu(x)} \check{\chi}(x)=\sum_{\ell<N}K_{j,\ell}(x) + r_{j,N}(x),
 \label{eq:lemN5.asymptotic-K}
\end{gather}
where we have set
\begin{gather*}
 K_{j,\ell}(x): = \frac{\big({\rm e}^{{\rm i}x_j}-1\big)\tilde{\nu}(x)^{\ell}}{|x|^{2\ell+2}} \check{\chi}(x) \qquad \text{and} \qquad r_{j,N}(x):= \frac{\big({\rm e}^{{\rm i}x_j}-1\big)\tilde{\nu}(x)^N}{|x|^{2N}\nu(x)}\check{\chi}(x).
\end{gather*}
By definition $K_{j,\ell}(x)$ and $r_{j,N}(x)$ are functions in $C^\infty\big(\R^n\setminus 0\big)$ with bounded support, and, near $x=0$ we have
\begin{gather*}
 K_{j,\ell}(x) =\frac{\big({\rm e}^{{\rm i}x_j}-1\big)\tilde{\nu}(x)^{\ell}}{|x|^{2\ell+2}} \qquad \text{and} \qquad r_{j,N}(x)= \frac{\big({\rm e}^{{\rm i}x_j}-1\big)\tilde{\nu}(x)^N}{|x|^{2N}\nu(x)}.
\end{gather*}

Recall that $\tilde{\nu}(x)$ is analytic and $\op{O}\big(|x|^4\big)$ near $x=0$. Thus, $\big({\rm e}^{{\rm i}x_j}-1\big)\tilde{\nu}(x)^{\ell}$ is smooth (and even analytic), and it is
$\op{O}\big(|x|^{4\ell +1}\big)$ near $x=0$. It then follows from Lemma~\ref{lemN4} that $\rho_{j,\ell}(\xi):=\widehat{K_{j,\ell}}(\xi)$ is a classical symbol on $\R^n$ of order
$-n-(4\ell+1)+2(\ell+1)=-n+1-2\ell$.

As mentioned above $\big({\rm e}^{{\rm i}x_j}-1\big)\tilde{\nu}(x)^{N}$ is analytic and $\op{O}\big(|x|^{4N+1}\big)$ near $x=0$. In addition, $|x|^{2N}\nu(x)$ is analytic and is equivalent to
$|x|^{2N+2}$ near $x=0$. It then follows from Lemma~\ref{lemN3} that the Fourier transform $\widehat{r}_{j,N}(\xi)$ is a standard symbol on $\R^n$ of order
 $-n-(4N+1)+ (2N+2)+1=-n+2-2N$. Combining this with~(\ref{eq:lemN5.asymptotic-K}) we see that, for all $N\geq 1$, we have
\begin{gather*}
 \rho_j(\xi)-\sum_{\ell<N} \rho_{j,\ell}(\xi)= \bigg( K - \sum_{\ell<N}K_{j,\ell}\bigg)^{\wedge}\! (\xi)=\widehat{r_{j,N}}(\xi)\in \stS^{-n+2-2N}\big(\R^n\big).
\end{gather*}
This shows that $\rho_j(\xi) \sim \sum\limits_{\ell \geq 0} \rho_{j,\ell}(\xi)$. As $\rho_{j,\ell}S^{-n+1+\ell}\big(\R^n\big)$ for $\ell\geq 0$, it then follows that $\rho_j(\xi)$ is a symbol in $S^{-n+1}\big(\R^n\big)$. This completes the proof of Lemma~\ref{lemN5}.
\end{proof}

\subsection*{Acknowlegements}
The research for this article was partially supported by the NSFC under Grant No.~11971328 (China). The author thanks Sylvie Paycha for discussions related to the subject matter of the paper. He also thanks University of Qu\'ebec of Montr\'eal (Canada) for its hospitality during the writing of the paper.

\pdfbookmark[1]{References}{ref}
\LastPageEnding


\begin{thebibliography}{99}
\footnotesize\itemsep=0pt

\bibitem{AG:AMS07}
Alinhac S., G\'erard P., Pseudo-differential operators and the {N}ash--{M}oser
 theorem, \textit{Graduate Studies in Mathematics}, Vol.~82, \href{https://doi.org/10.1090/gsm/082}{Amer. Math. Soc.},
 Providence, RI, 2007.

\bibitem{Ar:Springer81}
Arveson W., An invitation to {$C^*$}-algebras, \textit{Graduate Texts in
 Mathematics}, Vol.~39, \href{https://doi.org/10.1007/978-1-4612-6371-5}{Springer-Verlag}, New York~-- Heidelberg, 1976.

\bibitem{Ba:CRAS88-I}
Baaj S., Calcul pseudo-diff\'erentiel et produits crois\'es de
 {$C^*$}-alg\`ebres.~{I}, \textit{C.~R.~Acad. Sci. Paris S\'er.~I Math.}
 \textbf{307} (1988), 581--586.

\bibitem{Ba:CRAS88-II}
Baaj S., Calcul pseudo-diff\'erentiel et produits crois\'es de
 {$C^*$}-alg\`ebres.~{II}, \textit{C.~R.~Acad. Sci. Paris S\'er.~I Math.}
 \textbf{307} (1988), 663--666.

\bibitem{Co:CRAS80}
Connes A., {$C^{\ast} $} alg\`ebres et g\'eom\'etrie diff\'erentielle,
 \textit{C.~R.~Acad. Sci. Paris S\'er.~A-B} \textbf{290} (1980), A599--A604.

\bibitem{Co:AIM81}
Connes A., An analogue of the {T}hom isomorphism for crossed products of a
 {$C^{\ast} $}-algebra by an action of {${\bf R}$}, \href{https://doi.org/10.1016/0001-8708(81)90056-6}{\textit{Adv. Math.}}
 \textbf{39} (1981), 31--55.

\bibitem{Co:Foliations82}
Connes A., A survey of foliations and operator algebras, in Operator Algebras
 and Applications, {P}art~{I} ({K}ingston, {O}nt., 1980), \textit{Proc.
 Sympos. Pure Math.}, Vol.~38, Amer. Math. Soc., Providence, RI, 1982,
 521--628.

\bibitem{Co:NCG}
Connes A., Noncommutative geometry, Academic Press, Inc., San Diego, CA, 1994.

\bibitem{Co:Survey19}
Connes A., Noncommutative geometry, the spectral standpoint,
 \href{https://arxiv.org/abs/1910.10407}{arXiv:1910.10407}.

\bibitem{CM:AMS08}
Connes A., Marcolli M., Noncommutative geometry, quantum fields and motives,
 \textit{American Mathematical Society Colloquium Publications}, Vol.~55,
 Amer. Math. Soc., Providence, RI, 2008.

\bibitem{CM:GAFA95}
Connes A., Moscovici H., The local index formula in noncommutative geometry,
 \href{https://doi.org/10.1007/BF01895667}{\textit{Geom. Funct. Anal.}} \textbf{5} (1995), 174--243.

\bibitem{CM:JAMS14}
Connes A., Moscovici H., Modular curvature for noncommutative two-tori,
 \href{https://doi.org/10.1090/S0894-0347-2014-00793-1}{\textit{J.~Amer. Math. Soc.}} \textbf{27} (2014), 639--684,
 \href{https://arxiv.org/abs/1110.3500}{arXiv:1110.3500}.

\bibitem{CT:Baltimore11}
Connes A., Tretkoff P., The {G}auss--{B}onnet theorem for the noncommutative
 two torus, in Noncommutative Geometry, Arithmetic, and Related Topics, Johns
 Hopkins University Press, Baltimore, MD, 2011, 141--158, \href{https://arxiv.org/abs/0910.0188}{arXiv:0910.0188}.

\bibitem{DS:SIGMA15}
D\c{a}browski L., Sitarz A., An asymmetric noncommutative torus, \href{https://doi.org/10.3842/SIGMA.2015.075}{\textit{SIGMA}}
 \textbf{11} (2015), 075, 11~pages, \href{https://arxiv.org/abs/1406.4645}{arXiv:1406.4645}.

\bibitem{FGK:MPAG17}
Fathi A., Ghorbanpour A., Khalkhali M., Curvature of the determinant line
 bundle for the noncommutative two torus, \href{https://doi.org/10.1007/s11040-016-9234-9}{\textit{Math. Phys. Anal. Geom.}}
 \textbf{20} (2017), 4, 20~pages, \href{https://arxiv.org/abs/1410.0475}{arXiv:1410.0475}.

\bibitem{Fa:JMP15}
Fathizadeh F., On the scalar curvature for the noncommutative four torus,
 \href{https://doi.org/10.1063/1.4922815}{\textit{J.~Math. Phys.}} \textbf{56} (2015), 062303, 14~pages,
 \href{https://arxiv.org/abs/1410.8705}{arXiv:1410.8705}.

\bibitem{FK:JNCG12}
Fathizadeh F., Khalkhali M., The {G}auss--{B}onnet theorem for noncommutative
 two tori with a general conformal structure, \href{https://doi.org/10.4171/JNCG/97}{\textit{J.~Noncommut. Geom.}}
 \textbf{6} (2012), 457--480, \href{https://arxiv.org/abs/1005.4947}{arXiv:1005.4947}.

\bibitem{FK:JNCG13}
Fathizadeh F., Khalkhali M., Scalar curvature for the noncommutative two torus,
 \href{https://doi.org/10.4171/JNCG/145}{\textit{J.~Noncommut. Geom.}} \textbf{7} (2013), 1145--1183,
 \href{https://arxiv.org/abs/1110.3511}{arXiv:1110.3511}.

\bibitem{FK:JNCG15}
Fathizadeh F., Khalkhali M., Scalar curvature for noncommutative four-tori,
 \href{https://doi.org/10.4171/JNCG/198}{\textit{J.~Noncommut. Geom.}} \textbf{9} (2015), 473--503, \href{https://arxiv.org/abs/1301.6135}{arXiv:1301.6135}.

\bibitem{FK:Survey19}
Fathizadeh F., Khalkhali M., Curvature in noncommutative geometry, in Advances
 in Noncommutative Geometry, \href{https://doi.org/10.1007/978-3-030-29597-4_6}{Springer}, Cham, 2020, 321--420,
 \href{https://arxiv.org/abs/1901.07438}{arXiv:1901.07438}.

\bibitem{FW:JPDOA11}
Fathizadeh F., Wong M.W., Noncommutative residues for pseudo-differential
 operators on the noncommutative two-torus, \href{https://doi.org/10.1007/s11868-011-0030-9}{\textit{J.~Pseudo-Differ. Oper.
 Appl.}} \textbf{2} (2011), 289--302.

\bibitem{FGLS:JFA96}
Fedosov B.V., Golse F., Leichtnam E., Schrohe E., The noncommutative residue
 for manifolds with boundary, \href{https://doi.org/10.1006/jfan.1996.0142}{\textit{J.~Funct. Anal.}} \textbf{142} (1996),
 1--31.

\bibitem{FGK:JNCG19}
Floricel R., Ghorbanpour A., Khalkhali M., The {R}icci curvature in
 noncommutative geometry, \href{https://doi.org/10.4171/JNCG/324}{\textit{J.~Noncommut. Geom.}} \textbf{13} (2019),
 269--296, \href{https://arxiv.org/abs/1612.06688}{arXiv:1612.06688}.

\bibitem{GJP:MAMS17}
Gonz\'alez-P\'erez A.M., Junge M., Parcet J., Singular integrals in quantum
 {E}uclidean spaces, \textit{Mem. Amer. Math. Soc.}, {t}o appear,
 \href{https://arxiv.org/abs/1705.01081}{arXiv:1705.01081}.

\bibitem{Gu:AIM85}
Guillemin V., A new proof of {W}eyl's formula on the asymptotic distribution of
 eigenvalues, \href{https://doi.org/10.1016/0001-8708(85)90018-0}{\textit{Adv. Math.}} \textbf{55} (1985), 131--160.

\bibitem{Gu:AIM93}
Guillemin V., Gauged {L}agrangian distributions, \href{https://doi.org/10.1006/aima.1993.1064}{\textit{Adv. Math.}}
 \textbf{102} (1993), 184--201.

\bibitem{Gu:JFA93}
Guillemin V., Residue traces for certain algebras of {F}ourier integral
 operators, \href{https://doi.org/10.1006/jfan.1993.1096}{\textit{J.~Funct. Anal.}} \textbf{115} (1993), 391--417.

\bibitem{HLP:Part1}
Ha H., Lee G., Ponge R., Pseudodifferential calculus on noncommutative
 tori,~{I}. {O}scillating integrals, \href{https://doi.org/10.1142/S0129167X19500332}{\textit{Internat.~J. Math.}} \textbf{30}
 (2019), 1950033, 74~pages, \href{https://arxiv.org/abs/1803.03575}{arXiv:1803.03575}.

\bibitem{HLP:Part2}
Ha H., Lee G., Ponge R., Pseudodifferential calculus on noncommutative tori,
 {II}. {M}ain properties, \href{https://doi.org/10.1142/S0129167X19500344}{\textit{Internat.~J. Math.}} \textbf{30} (2019),
 1950034, 73~pages, \href{https://arxiv.org/abs/1803.03580}{arXiv:1803.03580}.

\bibitem{IM:JGP18}
Iochum B., Masson T., Heat asymptotics for nonminimal {L}aplace type operators
 and application to noncommutative tori, \href{https://doi.org/10.1016/j.geomphys.2018.02.014}{\textit{J.~Geom. Phys.}} \textbf{129}
 (2018), 1--24, \href{https://arxiv.org/abs/1707.09657}{arXiv:1707.09657}.

\bibitem{KV:GDEO}
Kontsevich M., Vishik S., Geometry of determinants of elliptic operators, in
 Functional Analysis on the Eve of the 21st Century, {V}ol.~1 ({N}ew
 {B}runswick, {NJ}, 1993), \textit{Progr. Math.}, Vol.~131, \href{https://doi.org/10.1007/978-1-4612-2582-9_6}{Birkh\"auser
 Boston}, Boston, MA, 1995, 173--197, \href{https://arxiv.org/abs/hep-th/9406140}{arXiv:hep-th/9406140}.

\bibitem{Lee:PhD}
Lee G., Pseudodifferential calculus on noncommutative tori. Resolvents and
 complex powers of elliptic operators, Ph.D.~Thesis, {S}eoul National
 University, 2018, available at
 \url{http://s-space.snu.ac.kr/handle/10371/143240}.

\bibitem{LP:Resolvent}
Lee G., Ponge R., Functional calculus for elliptic operators on noncommutative
 tori,~I, \href{https://doi.org/10.1007/s11868-020-00337-z}{\textit{J.~Pseudo-Differ. Oper. Appl.}}, {t}o appear,
 \href{https://arxiv.org/abs/1911.05500}{arXiv:1911.05500}.

\bibitem{LP:Powers}
Lee G., Ponge R., Functional calculus for elliptic operators on noncommutative
 tori,~II, {i}n preparation.

\bibitem{Le:AGAG99}
Lesch M., On the noncommutative residue for pseudodifferential operators with
 log-polyhomogeneous symbols, \href{https://doi.org/10.1023/A:1006504318696}{\textit{Ann. Global Anal. Geom.}} \textbf{17}
 (1999), 151--187, \href{https://arxiv.org/abs/dg-ga/9708010}{arXiv:dg-ga/9708010}.

\bibitem{LM:GAFA16}
Lesch M., Moscovici H., Modular curvature and {M}orita equivalence,
 \href{https://doi.org/10.1007/s00039-016-0375-6}{\textit{Geom. Funct. Anal.}} \textbf{26} (2016), 818--873, \href{https://arxiv.org/abs/1505.00964}{arXiv:1505.00964}.

\bibitem{LM:Survey19}
Lesch M., Moscovici H., Modular {G}aussian curvature, in Advances in
 Noncommutative Geometry, \href{https://doi.org/10.1007/978-3-030-29597-4_8}{Springer}, Cham, 2020, 463--490, \href{https://arxiv.org/abs/1810.10394}{arXiv:1810.10394}.

\bibitem{LN:JNCG13}
Lesch M., Neira~Jim\'enez C., Classification of traces and hypertraces on
 spaces of classical pseudodifferential operators, \href{https://doi.org/10.4171/JNCG/123}{\textit{J.~Noncommut.
 Geom.}} \textbf{7} (2013), 457--498, \href{https://arxiv.org/abs/1011.3238}{arXiv:1011.3238}.

\bibitem{LP:PLMS07}
Lescure J.M., Paycha S., Uniqueness of multiplicative determinants on elliptic
 pseudodifferential operators, \href{https://doi.org/10.1112/plms/pdm004}{\textit{Proc. Lond. Math. Soc.}} \textbf{94} (2007), 772--812.

\bibitem{LNP:TAMS16}
L\'evy C., Neira~Jim\'enez C., Paycha S., The canonical trace and the
 noncommutative residue on the noncommutative torus, \href{https://doi.org/10.1090/tran/6369}{\textit{Trans. Amer.
 Math. Soc.}} \textbf{368} (2016), 1051--1095, \href{https://arxiv.org/abs/1303.0241}{arXiv:1303.0241}.

\bibitem{Liu:JNCG18}
Liu Y., Modular curvature for toric noncommutative manifolds,
 \href{https://doi.org/10.4171/JNCG/285}{\textit{J.~Noncommut. Geom.}} \textbf{12} (2018), 511--575,
 \href{https://arxiv.org/abs/1510.04668}{arXiv:1510.04668}.

\bibitem{Liu:arXiv18b}
Liu Y., Hypergeometric function and modular curvature~II. {C}onnes--{M}oscovici
 functional relation after {L}esch's work, \href{https://arxiv.org/abs/1811.07967}{arXiv:1811.07967}.

\bibitem{Liu:arXiv20}
Liu Y., General rearrangement lemma for heat trace asymptotic on noncommutative
 tori, \href{https://arxiv.org/abs/2004.05714}{arXiv:2004.05714}.

\bibitem{MSS:PAMS08}
Maniccia L., Schrohe E., Seiler J., Uniqueness of the {K}ontsevich--{V}ishik
 trace, \href{https://doi.org/10.1090/S0002-9939-07-09168-X}{\textit{Proc. Amer. Math. Soc.}} \textbf{136} (2008), 747--752,
 \href{https://arxiv.org/abs/math.FA/0702250}{arXiv:math.FA/0702250}.

\bibitem{Pa:Private}
Paycha S., {P}rivate communication.

\bibitem{PS:GAFA07}
Paycha S., Scott S., A {L}aurent expansion for regularized integrals of
 holomorphic symbols, \href{https://doi.org/10.1007/s00039-007-0597-8}{\textit{Geom. Funct. Anal.}} \textbf{17} (2007),
 491--536, \href{https://arxiv.org/abs/math.AP/0506211}{arXiv:math.AP/0506211}.

\bibitem{Po:JFA07}
Ponge R., Noncommutative residue for {H}eisenberg manifolds. {A}pplications in
 {CR} and contact geometry, \href{https://doi.org/10.1016/j.jfa.2007.07.001}{\textit{J.~Funct. Anal.}} \textbf{252} (2007),
 399--463, \href{https://arxiv.org/abs/math.DG/0607296}{arXiv:math.DG/0607296}.

\bibitem{Po:JAM10}
Ponge R., Traces on pseudodifferential operators and sums of commutators,
 \href{https://doi.org/10.1007/s11854-010-0001-8}{\textit{J.~Anal. Math.}} \textbf{110} (2010), 1--30, \href{https://arxiv.org/abs/0707.4265}{arXiv:0707.4265}.

\bibitem{Po:JMP20}
Ponge R., Connes's trace theorem for curved noncommutative tori: application to
 scalar curvature, \href{https://doi.org/10.1063/5.0005052}{\textit{J.~Math. Phys.}} \textbf{61} (2020), 042301, 27~pages, \href{https://arxiv.org/abs/1912.07113}{arXiv:1912.07113}.

\bibitem{Ri:PJM81}
Rieffel M.A., {$C^{\ast} $}-algebras associated with irrational rotations,
 \href{https://doi.org/10.2140/pjm.1981.93.415}{\textit{Pacific~J. Math.}} \textbf{93} (1981), 415--429.

\bibitem{Ri:CM90}
Rieffel M.A., Noncommutative tori~-- a case study of noncommutative
 differentiable manifolds, in Geometric and Topological Invariants of Elliptic
 Operators ({B}runswick, {ME}, 1988), \textit{Contemp. Math.}, Vol.~105, \href{https://doi.org/10.1090/conm/105/1047281}{Amer.
 Math. Soc.}, Providence, RI, 1990, 191--211.

\bibitem{RT:Birkhauser10}
Ruzhansky M., Turunen V., Pseudo-differential operators and symmetries.
 Background analysis and advanced topics, \textit{Pseudo-Differential
 Operators. Theory and Applications}, Vol.~2, \href{https://doi.org/10.1007/978-3-7643-8514-9}{Birkh\"auser Verlag}, Basel, 2010.

\bibitem{SZ:arXiv19}
Sukochev F., Zanin D., Local invariants of non-commutative tori,
 \href{https://arxiv.org/abs/1910.00758}{arXiv:1910.00758}.

\bibitem{Ta:JPCS18}
Tao J., The theory of pseudo-differential operators on the noncommutative
 $n$-torus, \href{https://doi.org/10.1088/1742-6596/965/1/012042}{\textit{J.~Phys. Conf. Ser.}} \textbf{965} (2018), 012042,
 12~pages, \href{https://arxiv.org/abs/1704.02507}{arXiv:1704.02507}.

\bibitem{Ta:PDE2}
Taylor M.E., Partial differential equations. {II}.~Qualitative studies of
 linear equations, \textit{Applied Mathematical Sciences}, Vol.~116,
 \href{https://doi.org/10.1007/978-1-4757-4187-2}{Springer-Verlag}, New York, 1996.

\bibitem{Wo:HDR}
Wodzicki M., Spectral asymmetry and noncommutative residue, {H}abilitation
 Thesis, Steklov Institute, Moscow, 1984.

\bibitem{Wo:NCR}
Wodzicki M., Noncommutative residue. {I}.~{F}undamentals, in {$K$}-Theory,
 Arithmetic and Geometry ({M}oscow, 1984--1986), \textit{Lecture Notes in
 Math.}, Vol.~1289, \href{https://doi.org/10.1007/BFb0078372}{Springer}, Berlin, 1987, 320--399.

\end{thebibliography}
\end{document}